\documentclass[12pt]{amsart}
\usepackage{amsmath, amssymb, amsthm} 
\usepackage{hyperref} 
\usepackage{geometry} 
\usepackage{color}
\usepackage{tikz}
\usepackage{graphicx}
\usepackage{float}
\newtheorem{theorem}{Theorem}[section]
\newtheorem{lemma}[theorem]{Lemma}
\newtheorem{corollary}[theorem]{Corollary}
\newtheorem{proposition}[theorem]{Proposition}

\theoremstyle{definition}
\newtheorem{definition}[theorem]{Definition}

\theoremstyle{remark}
\newtheorem{remark}[theorem]{Remark}

\newcommand{\R}{\mathbb{R}}
\newcommand{\C}{\mathbb{C}}
\newcommand{\Hk}{H_k}
\newcommand{\dist}{\operatorname{dist}}
\newcommand{\Vol}{\operatorname{Vol}}
\newcommand{\bla}{\boldsymbol{\lambda}}
\newcommand{\pr}{\partial}
\newcommand{\ind}{\operatorname{ind}}
\newcommand{\dv}{\operatorname{div}}
\newcommand{\dd}{\,d}
\newcommand{\Mass}{\mathbb{M}}
\newcommand{\II}{\mathrm{I}\!\mathrm{I}}

\title{Stability and Area-Minimizing Property of Higher-Dimensional Helicoids}

\author{Chung-Jun Tsai}
\address{Department of Mathematics, National Taiwan University, and National Center for Theoretical Sciences, Math Division, Taipei 10617, Taiwan}
\email{cjtsai@ntu.edu.tw}

\author{Mao-Pei Tsui}
\address{Department of Mathematics, National Taiwan University, and National Center for Theoretical Sciences, Math Division, Taipei 10617, Taiwan}
\email{maopei@math.ntu.edu.tw}

\author{Jingbo Wan}
\address{Laboratoire Jacques-Louis Lions de Sorbonne Universit\'e, 4 place Jussieu, Paris 75005, France}
\email{jingbo.wan@sorbonne-universite.fr}

\author{Mu-Tao Wang}
\address{Department of Mathematics, Columbia University, New York, NY 10027, USA}
\email{mtwang@math.columbia.edu}

\thanks{C.-J.~Tsai is supported in part by the National Science and Technology Council grant 112-2628-M-002-004-MY4.  M.-P.~Tsui is supported in part by the National Science and Technology Council grants 
115-2115-M-002-005-MY3. J.~Wan is supported by ERC-2023 AdG 101141855 BLaHST. M.-T.~Wang is supported in part by the National Science Foundation under Grant DMS-2404945. This work was supported by grants from the Simons Foundation [MPS-TSM-00007411, M.-T.~W.]. Part of this work was carried out when M.-T.~Wang was visiting the National Center for Theoretical Sciences.} 

\date{}

\begin{document}

\begin{abstract}
For each integer $k\geq 1$, we study the $(k+1)$-dimensional helicoid $H_k\subset\R^{2k+1}$ parametrized by
\[
  (u_1,\ldots,u_k,s)
  \longmapsto
  \bigl(u_1e^{is},\ldots,u_ke^{is},s\bigr)
  \in \C^k\times\R\cong \R^{2k+1}.
\] These helicoids form a basic and distinguished family of complete, properly embedded minimal submanifolds diffeomorphic to $\mathbb R^{k+1}$, and provide natural higher-dimensional analogues of the classical helicoid in $\mathbb R^3$.

We completely determine their stability: $H_k$ is stable for $k\geq 3$ and unstable for $k\leq 2$.  The sharp transition at $k=3$ is particularly striking: while the classical helicoid $(k=1)$ and its first higher-dimensional analogue $(k=2)$ are unstable, the four-dimensional helicoid $H_3\subset\R^7$ is already stable.

For $k\geq 3$, we also determine their area-minimizing property: $H_k$ is area-minimizing when $k$ is even and not area-minimizing when $k$ is odd. The area-minimizing result is proved by constructing an explicit calibration, while the non-area-minimizing result follows from an explicit competitor. In particular, for every even $k\geq 4$, the $(k+1)$-dimensional helicoid $H_k$ is an entire minimal graph in $\mathbb{R}^{2k+1}$ that is area-minimizing. 
\end{abstract}

\maketitle


\section{Introduction}

The classical helicoid, first described by Meusnier in 1776, is one of the most fundamental models in the theory of minimal surfaces. The work of Colding and Minicozzi shows, in a precise sense, that embedded minimal disks in a ball in $\R^3$ are modeled either on planes or on helicoids \cite{ColdingMinicozzi}. Meeks and Rosenberg proved that the plane and the helicoid are the only complete, properly embedded, simply connected minimal surfaces in $\R^3$ \cite{MeeksRosenberg}. Bernstein and Breiner subsequently showed that every complete, non-flat, properly embedded minimal surface in $\R^3$ with finite genus and one end is asymptotic to a helicoid \cite{BernsteinBreiner}. These results underscore the central role of the helicoid in the local and global geometry of embedded minimal surfaces.

Despite this rigidity, the classical helicoid is unstable: it admits a compactly supported normal variation with negative second variation of area.  The pioneering work of Fischer-Colbrie and Schoen \cite{FCS} (see also do Carmo-Peng \cite{dCP} and references in \cite{FCS} for earlier results) on stable minimal surfaces shows that a complete stable minimal surface in $\mathbb{R}^3$ must be a plane. Their arguments imply the instability of the helicoid indirectly through this rigidity theorem. 

The above contrast motivates the study of higher-dimensional helicoids.  For any integer $k\geq 1$, let $H_k$ be the image of the proper embedding
\begin{align}\label{eq:helicoid-parametrization}\begin{split}
  \begin{array}{cccl}
    &\R^k\times\R &\to &\C^k\times\R\cong\R^{2k+1}, \\
    &(u_1,\ldots,u_k,s) &\mapsto & \bigl(u_1e^{is},\ldots,u_ke^{is},s\bigr).
  \end{array}
\end{split}\end{align}
It is not hard to see that $H_k$ is a $(k+1)$-dimensional properly embedded minimal submanifold of $\R^{2k+1}$, and $H_1$ is the classical helicoid.  This family belongs to the broader class of higher-dimensional ruled minimal submanifolds studied in \cite{BarbosaDajczerJorge}.

Our first main theorem gives a complete and unexpectedly sharp answer to the
stability question.

\begin{theorem}\label{thm:stability} 
The $(k+1)$-dimensional helicoid $H_k$ is stable if and only if $k\geq 3$.
Namely,
\[
  H_k \text{ is unstable for } k\leq 2,
  \qquad
  H_k \text{ is stable for } k\geq 3.
\]
\end{theorem}

The proof proceeds by first showing that the stability condition is equivalent to 
the functional inequality \eqref{eq:stability-cond} for vector-valued functions on $\R^k$ and then establishing this inequality when $k\geq 3$.  When $k\leq 2$, the instability is proved by constructing unstable perturbations.

The dichotomy in Theorem~\ref{thm:stability} is perhaps the most surprising feature of the result.  One might expect the instability of the classical helicoid to persist through several dimensions, or even throughout the whole family.  Instead, stability appears abruptly at $k=3$: the three-dimensional $H_2\subset\R^5$ is unstable, whereas the four-dimensional $H_3\subset\R^7$ is stable.  

For $k\geq 3$, we consider the substantially stronger property of area minimization.  Stability is infinitesimal: it asserts the nonnegativity of the second variation under compactly supported normal variations.  Area minimization is global, requiring comparison with arbitrary compactly supported competitors having the same boundary.  Our second main theorem shows that, for $k\geq 3$, the area-minimizing property of $H_k$ is completely determined by the parity of $k$.

\begin{theorem}\label{thm:area-minimizing} 
For $k\geq 3$, the $(k+1)$-dimensional helicoid $H_k$ is area-minimizing when $k$ is even, and not area-minimizing when $k$ is odd. 
\end{theorem}

The area-minimizing result follows from the construction of an explicit calibration in the sense of Harvey and Lawson \cite{HarveyLawson}. The construction of this calibration is closely related to the calibrations developed by Lawlor \cite{LawlorCriterion} and by Kerckhove and Lawlor \cite{KerckhoveLawlor} for the determinantal variety $C(2,k,1)$; see also Morgan's survey \cite{MorganCalibrations} for a broader account of calibration methods. There is a uniform construction of the calibration for all even $k\geq 6$. However, this construction breaks down when $k=4$ due to the failure of a key lemma (Lemma \ref{lem:trig-inequality}). To handle this exceptional case, we exploit the quaternionic structure in dimension four to construct a different calibration.

When $k\geq 3$ is odd, we identify a portion of $H_k$ and construct a competitor with the same boundary but strictly smaller area, thereby showing that $H_k$ is not area-minimizing.

Theorems~\ref{thm:stability} and \ref{thm:area-minimizing} together exhibit a striking dependence on dimension: $H_k$ becomes stable at $k=3$, while its area-minimization property for $k\geq 3$ is determined entirely by the parity of $k$. 

Finally, we study the broader family of generalized helicoids introduced by Barbosa--Dajczer--Jorge and by Bryant \cite{BarbosaDajczerJorge,BryantAustere}.   For positive real parameters $\bla=(\lambda_1,\ldots,\lambda_k)$, we write $\mathcal{H}_{\bla}$ for the generalized helicoids parametrized by
\begin{equation}\label{eq:general-helicoid}
  (u_1,\ldots,u_k,s)\mapsto \bigl(u_1e^{i\lambda_1s},\ldots,
         u_ke^{i\lambda_ks},s\bigr),
\end{equation}
so that $H_k$ is the case $\lambda_1=\cdots=\lambda_k=1$. In our earlier paper \cite{TsaiTsuiWanWang}, we obtained a sufficient condition, stated in \eqref{eq:entire-criterion}, for a generalized helicoid to be realized as an entire minimal graph. Here we prove that this condition is also necessary.  In particular, we conclude that entireness of $H_k$ is again determined by the parity of $k$: $H_k$ is entire if and only if $k$ is even. Regarding the stability of generalized helicoids, we show that, in the moduli space of generalized helicoids, the unstable members form an open subset.  We also give various criteria, stated directly in terms of the parameters $\lambda_1,\ldots,\lambda_k$, that distinguish stable generalized helicoids from unstable ones. 

For the family of $(k+1)$-dimensional helicoids $H_k$, the preceding results are summarized in Table \ref{tab:helicoid_properties}.

\begin{table}[H]
\centering
\renewcommand{\arraystretch}{1.2}
\begin{tabular}{cccc}
\hline
$k$  & \textbf{Stability} & \textbf{Area-Minimizing} & \textbf{Entire} \\
\hline
$1$               & No  & No  & No \\
\hline
$2$                & No  & No & Yes\\
\hline
Odd $k \ge 3$                  & Yes & No & No\\
\hline
Even $k \ge 4$    & Yes & Yes & Yes \\
\hline
\end{tabular}
\caption{Variational and global properties of $H_k \subset \mathbb{R}^{2k+1}$}
\label{tab:helicoid_properties}
\end{table}
In particular, we obtain the following. 
\begin{corollary}
    For every even $k\geq 4$, the $(k+1)$-dimensional helicoid $H_k$ is an entire minimal graph in $\mathbb{R}^{2k+1}$ that is area-minimizing. 
\end{corollary}

Section \ref{sec:stability-reduction} reduces the stability problem for generalized helicoids to a functional inequality for vector-valued functions on $\R^k$. Section \ref{sec:Hk-stability/instability/entireness} proves this functional inequality to establish the stability of $\Hk$ for $k\geq 3$. Sections \ref{sec:calibration-Hk-even} and \ref{sec:calibration-H4} prove the area-minimizing property of $H_k$ for even $k\geq 6$ and of $H_4$ by constructing respective calibrations, while Section \ref{sec:competitor-Hk-odd} proves that $H_k$ is not area-minimizing for odd $k$ by constructing explicit competitors. Finally, Section \ref{sec:Hardy-anisotropic} identifies a stable region in the moduli space of generalized helicoids, Section \ref{sec:unstable} develops the instability theory, including the instability of $H_2$ and the openness of the unstable region, while Section \ref{sec:entireness} studies the entireness problem and determines the entireness of $H_k$ for all $k$.

\medskip

{\it Disclosure of AI use:} The authors used AI tools, including ChatGPT 5.6 Sol and Claude Fable 5, to assist with literature searches, explore ideas, and perform symbolic calculations in the preparation of this work. All proofs and calculations, however, were independently verified by the authors.

\section{Reduction of the stability condition for generalized helicoids}\label{sec:stability-reduction}

In this section, we study the second variation of generalized helicoids. In particular, we reduce the stability condition to a functional inequality for vector-valued functions on $\R^k$. Identify $\mathbb R^{2k+1}$ with $\mathbb C^k \times \mathbb R$. Denote $\operatorname{diag}(\bla)$ by $\Lambda$.  The parametrization \[F_{\bla}(u,s) = \bigl(e^{i\Lambda s}u, \, s\bigr)\] of generalized helicoid $\mathcal{H}_{\bla}$ in \eqref{eq:general-helicoid} has derivatives
\[
 \pr_{u_j} F_{\bla} = \bigl(e^{i\lambda_j s}\mathbf{e}_j, \, 0\bigr), \qquad
 \pr_s F_{\bla} = \bigl(i\Lambda e^{i\Lambda s} u, \, 1\bigr).
\]
Set $h^2 = 1 + |\Lambda u|^2$, where $|\Lambda u|^2 = \sum_{j=1}^k(\lambda_j u_j)^2$.
An orthonormal tangent frame on $\mathcal{H}_{\bla}$ is $E_j = \pr_{u_j} F_{\bla}$ ($1 \le j \le k$) and $E_{k+1} = h^{-1}\pr_s F_{\bla}$, giving
\begin{equation*}
 g = \sum_{j=1}^k du_j^2 + h^2 ds^2, \qquad d\mu = h \, du \, ds.
\end{equation*}
Second derivatives satisfy
\[
 \pr_{u_j u_j}^2 F_{\bla} = 0, \qquad
 \pr_{ss}^2 F_{\bla} = \bigl(-\Lambda^2 e^{i\Lambda s} u, \, 0\bigr) = -\sum_{j=1}^k \lambda_j^2 u_j \pr_{u_j} F_{\bla} \in T\mathcal{H}_{\bla}.
\]
Since $g^{js} = 0$, the mean curvature of  $\mathcal{H}_{\bla}$ vanishes:
\[
 H = \left(\sum_{j=1}^k \pr_{u_j u_j}^2 F_{\bla} + h^{-2}\pr_{ss}^2 F_{\bla}\right)^\perp = 0.
\]

We now study the stability of $\mathcal{H}_{\bla}$ as a minimal submanifold. Recall that the second variation of area associated with a normal vector field $V$ is given by

\begin{equation*}
 Q(V,V) = \int_{\mathcal{H}_{\bla}}  \sum_{a=1}^{k+1}|\nabla^\perp_{E_a} V|^2-\sum_{a,b=1}^{k+1} \langle \II(E_a, E_b), V \rangle^2,
\end{equation*} 
where $\{E_a\}_{a=1}^{k+1}$ is any local orthonormal tangent frame and $\II$ denotes the second fundamental form. The minimal submanifold $\mathcal{H}_{\bla}$ is stable if and only if $Q(V, V)\geq 0$ for every compactly supported normal field $V$. It turns out that the nonnegativity of $Q$ is equivalent to the nonnegativity of $q_{\bla}$ in the following definition.

\begin{definition}
    For $ f=(f^1,\ldots, f^k) \in C_c^\infty(\mathbb R^k; \mathbb R^k)$, define
\begin{equation}\label{eq:L^2-functional}
 q_{\bla}[f] = \int_{\mathbb R^k} h \left[
 \sum_{j=1}^k |\nabla f^j|^2
 + \left|\nabla\bigl(\Lambda u\cdot  f\bigr)\right|^2
 - \frac{3|\Lambda  f|^2}{h^2}
 \right], 
\end{equation}
where $u_1, \ldots, u_k$ are standard coordinates on $\mathbb{R}^k$, $\Lambda u\cdot f$ is the scalar function $\sum_{j=1}^k \lambda_j u_j f^j$, $|\Lambda f|^2=\sum_{j=1}^k (\lambda_j f^j)^2$,  $h = \sqrt{1 + |\Lambda u|^2}=\sqrt{1 +\sum_{j=1}^k (\lambda_j u_j)^2 }$, and $\nabla$ is the standard gradient operator for scalar functions on $\mathbb{R}^k$.

\end{definition}

\begin{proposition}\label{prop:stability-reduction} 
The generalized helicoid $\mathcal{H}_{\bla}$ is stable if and only if 
\begin{equation}\label{eq:stability-cond}
 q_{\bla}[f] \ge 0
 \quad\forall  f \in C_c^\infty(\mathbb R^k; \mathbb R^k).
\end{equation}
Moreover, if \eqref{eq:stability-cond} fails, then $\ind \mathcal{H}_{\bla} = \infty$.
\end{proposition}
\begin{proof}
\smallskip\noindent\emph{Step 1: Parametrization of normal vector fields.}
Given $\tilde f \in C_c^\infty(\mathbb R^k \times \mathbb R; \mathbb R^k)$, define
\begin{equation}\label{eq:normal-field}
 V(u,s) = \bigl(i e^{i\Lambda s}\tilde f(u,s), \, -(\Lambda u)\cdot \tilde f(u,s)\bigr) \in \mathbb C^k \times \mathbb R.
\end{equation}
Taking ambient inner products with the tangent vectors gives
\[
 \langle V, \pr_{u_j} F_{\bla} \rangle = \operatorname{Re}\langle i e^{i\lambda_j s}\tilde f^j, e^{i\lambda_j s}\rangle_{\mathbb C} = 0,
\]
\[
 \langle V, \pr_s F_{\bla} \rangle = \operatorname{Re}\langle i e^{i\Lambda s}\tilde f, i\Lambda e^{i\Lambda s}u\rangle_{\mathbb C^k} - (\Lambda u)\cdot \tilde f = 0.
\]
Thus $V \in T^\perp \mathcal{H}_{\bla}$ is normal. Since $\dim(T^\perp \mathcal{H}_{\bla}) = k$, the representation \eqref{eq:normal-field} uniquely parametrizes all compactly supported normal fields. 

\smallskip\noindent\emph{Step 2: Computation of the second fundamental form.}
Differentiating the parametrization yields
\[
 \pr_{u_j u_m}^2 F_{\bla} = 0, \qquad \pr_{ss}^2 F_{\bla} \in T\mathcal{H}_{\bla}, \qquad \pr_{u_j s}^2 F_{\bla} = \bigl(i\lambda_j e^{i\lambda_j s}\mathbf{e}_j, \, 0\bigr).
\]
Taking normal projections gives $\II(E_j, E_m) = 0$ and $\II(E_{k+1}, E_{k+1}) = 0$. For the cross terms,
\[
 \langle \II(E_j, E_{k+1}), V \rangle = h^{-1} \langle \pr_{u_j s}^2 F_{\bla}, V \rangle = h^{-1} \operatorname{Re}\langle i\lambda_j e^{i\lambda_j s}\mathbf{e}_j, i e^{i\lambda_j s}\tilde f^j \rangle = \frac{\lambda_j \tilde f^j}{h}.
\]
Therefore, 
\begin{equation}\label{eq:II-squared}
 \sum_{a,b=1}^{k+1} \langle \II(E_a, E_b), V \rangle^2 = 2\sum_{j=1}^k \langle \II(E_j, E_{k+1}), V \rangle^2 = \frac{2|\Lambda \tilde f|^2}{h^2}.
\end{equation}

\smallskip\noindent\emph{Step 3: Derivatives along $u_j$.}
Differentiating $V$ along $\pr_{u_j}$ gives
\[
 \overline\nabla_{\pr_{u_j}} V = \bigl(i e^{i\Lambda s}\pr_{u_j} \tilde f, \, -\pr_{u_j}\bigl((\Lambda u)\cdot \tilde f\bigr)\bigr).
\]
Its tangential components are
\[
 \langle \overline\nabla_{\pr_{u_j}} V, E_m \rangle = 0 \quad (1 \le m \le k),
\]
\[
 \langle \overline\nabla_{\pr_{u_j}} V, E_{k+1} \rangle = h^{-1}\left( (\Lambda u)\cdot \pr_{u_j} \tilde f - \pr_{u_j}\bigl((\Lambda u)\cdot \tilde f\bigr) \right) = -h^{-1}\lambda_j \tilde f^j.
\]
Subtracting the tangential component from the ambient norm $|\overline\nabla_{\pr_{u_j}} V|^2$ gives
\begin{equation}\label{eq:u-deriv}
 |\nabla^\perp_{E_j} V|^2 = |\overline\nabla_{\pr_{u_j}} V|^2 - \langle \overline\nabla_{\pr_{u_j}} V, E_{k+1} \rangle^2 = |\pr_{u_j} \tilde f|^2 + \left|\pr_{u_j}\bigl((\Lambda u)\cdot \tilde f\bigr)\right|^2 - \frac{\lambda_j^2(\tilde f^j)^2}{h^2}.
\end{equation}

\smallskip\noindent\emph{Step 4: Derivative along $s$.}
Differentiating $V$ along $\pr_s$ gives the orthogonal decomposition into normal and tangential parts:
\[
 \overline\nabla_{\pr_s} V = \underbrace{\bigl(i e^{i\Lambda s}\pr_s \tilde f, \, -(\Lambda u)\cdot \pr_s \tilde f\bigr)}_{\in T^\perp \mathcal{H}_{\bla}} - \underbrace{\sum_{j=1}^k \lambda_j \tilde f^j \pr_{u_j} F_{\bla}}_{\in T \mathcal{H}_{\bla}}.
\]
Taking the norm of the normal component and scaling by $E_{k+1} = h^{-1}\pr_s$ gives
\begin{equation}\label{eq:s-deriv}
 |\nabla^\perp_{E_{k+1}} V|^2 = h^{-2}\left|(\overline\nabla_{\pr_s} V)^\perp\right|^2 = \frac{|\pr_s \tilde f|^2 + \bigl((\Lambda u)\cdot \pr_s \tilde f\bigr)^2}{h^2}.
\end{equation}

\smallskip\noindent\emph{Step 5: Second variation and stability.}
Combining \eqref{eq:u-deriv} and \eqref{eq:s-deriv}, we obtain
\[
 |\nabla^\perp V|^2 = \sum_{j=1}^k |\pr_{u_j} \tilde f|^2 + \left|\nabla_u\bigl((\Lambda u)\cdot \tilde f\bigr)\right|^2 + \frac{|\pr_s \tilde f|^2 + \bigl((\Lambda u)\cdot \pr_s \tilde f\bigr)^2 - |\Lambda \tilde f|^2}{h^2}.
\]
Subtracting \eqref{eq:II-squared} and integrating with respect to $d\mu = h \, du \, ds$ yields
\begin{equation*}
 Q(V,V) = \int_{\mathbb R} q_{\bla}[\tilde f(\cdot, s)] \, ds + \int_{\mathbb R^{k+1}} h^{-1} \left[|\pr_s \tilde f|^2 + \bigl((\Lambda u)\cdot \pr_s \tilde f\bigr)^2\right] du \, ds.
\end{equation*}
If $q_{\bla}[ f] \ge 0$ for all $ f \in C_c^\infty(\mathbb R^k; \mathbb R^k)$, then $Q(V,V) \ge 0$, establishing stability.

Conversely, if $q_{\bla}[f_0] < 0$ for some $ f_0$, choose $0 \ne \chi \in C_c^\infty(-1,1)$ and set 
\[\tilde f_m(u,s) = \chi\left(\frac{s - 3mL}{L}\right) f_0(u)\]
for $m \in \mathbb Z$. Then for $V_m = \bigl(i e^{i\Lambda s}\tilde f_m, \, -(\Lambda u)\cdot \tilde f_m\bigr)$
\[
 Q(V_m,V_m) = L \|\chi\|_2^2 q_{\bla}[ f_0] + L^{-1}\|\chi'\|_2^2 \int_{\mathbb R^k} h^{-1}\left[| f_0|^2 + \bigl((\Lambda u)\cdot f_0\bigr)^2\right] du.
\]
For sufficiently large $L$, $Q(V_m,V_m) < 0$. Since $\{V_m\}_{m\in\mathbb Z}$ have disjoint compact supports in $s$, they span an infinite-dimensional negative subspace, proving $\ind \mathcal{H}_{\bla} = \infty$.
\end{proof}

\section{Stability of \texorpdfstring{$H_k$}{Hk} for every \texorpdfstring{$k\geq3$}{k >=3}}\label{sec:Hk-stability/instability/entireness}

We prove the following sharp functional inequality for vector-valued functions on $\R^k$ for every $k\ge3$.

\begin{proposition} \label{prop:sharp-hardy}
    Let $k\ge3$, let $x=(x_1,\ldots,x_k)$ be the standard coordinates on $\R^k$, and set
    \[
        h(x)=\sqrt{1+|x|^2}, \qquad c_k=\frac{(k-1)^2+8}{4}.
    \]
    Then every $f=(f^1,\ldots,f^k)\in C_c^\infty(\R^k;\R^k)$ satisfies
    \begin{equation}\label{eq:hardy-sharp}
       \mathcal Q[f]:= \int_{\R^k}h \left(\sum_{j=1}^k|\nabla f^j|^2 + |\nabla(x\cdot f)|^2 - c_k\frac{|f|^2}{h^2} \right)\ge0.
    \end{equation}

\end{proposition}

The proof separates the radial and spherical behavior of $f$. We begin with the spherical estimates.

We write $r=|x|$ and $x=r\omega$, where $\omega\in S^{k-1}$.  On $S^{k-1}$, let $\nabla_S$, $\dv_S$, and $D$ denote the gradient, divergence, and Levi--Civita connection, respectively.  If $F$ is tangent to $S^{k-1}$, let $F^\flat$ be its metric-dual one-form.  We use the convention
\[
 |dF^\flat|^2
 =\sum_{i<j}(D_iF_j-D_jF_i)^2
\]
in a local orthonormal frame.

\begin{lemma}
\label{lem:sphere}
Let $k\ge3$.  First, every smooth function $u:S^{k-1}\to\R$ with zero spherical mean satisfies
\begin{equation}\label{eq:scalar-gap}
\int_{S^{k-1}}|\nabla_Su|^2\dd\omega
\ge
(k-1)\int_{S^{k-1}}u^2\dd\omega.
\end{equation}
Second, every smooth tangent vector field $F$ on $S^{k-1}$ satisfies
\begin{equation}\label{eq:oneform-gap}
\int_{S^{k-1}}
\left((\dv_SF)^2+|dF^\flat|^2\right)\dd\omega
\ge
(k-1)\int_{S^{k-1}}|F|^2\dd\omega.
\end{equation}

\end{lemma}

\begin{proof}
The scalar estimate \eqref{eq:scalar-gap} is the first nonzero eigenvalue identity $\lambda_1(S^{k-1})=k-1$. Let $\eta=F^\flat$.  By the Bochner--Weitzenb\"ock formula
for one-forms \cite[Chapter~3]{LiGA},
\begin{equation}\label{eq:bochner}
 \int_{S^{k-1}}\bigl(|d\eta|^2+|\delta\eta|^2\bigr)\dd\omega
 =\int_{S^{k-1}}\bigl(|\nabla\eta|^2+(k-2)|\eta|^2\bigr)\dd\omega.
\end{equation}
Decompose $\nabla\eta=S+A$ into its symmetric and antisymmetric parts.  Our normalization gives $|A|^2=\tfrac12|d\eta|^2$, while $(\operatorname{tr}S)^2\le(k-1)|S|^2$.  Since $\delta\eta=-\dv_SF$, it follows that
\[
 |\nabla\eta|^2\ge
 \frac1{k-1}|\delta\eta|^2+\frac12|d\eta|^2.
\]
Substitution in \eqref{eq:bochner}, followed by rearrangement, gives
\begin{equation}\label{eq:strong-oneform}
 \int_{S^{k-1}}\left(
 |\delta\eta|^2+\frac{k-1}{2(k-2)}|d\eta|^2
 \right)\dd\omega
 \ge(k-1)\int_{S^{k-1}}|\eta|^2\dd\omega.
\end{equation}
For $k\ge3$, $(k-1)/(2(k-2))\le1$, so \eqref{eq:strong-oneform} implies \eqref{eq:oneform-gap}. 
\end{proof}

\begin{proof}[Proof of Proposition~\ref{prop:sharp-hardy}]
For $r>0$, decompose
\begin{equation*}
 f(r\omega)=a(r,\omega)\omega+F(r,\omega),
 \qquad F(r,\omega)\in T_\omega S^{k-1}.
\end{equation*}
Then $x\cdot f=ra$ and $|f|^2=a^2+|F|^2$. The functions $a=\omega\cdot f(r\omega)$ and $F=f(r\omega)-a\omega$ are smooth on $[0,\infty)\times S^{k-1}$ and vanish for sufficiently large $r$. Hence all radial boundary terms vanish: at infinity by compact support and at the origin because they contain a factor $r^k$.

A calculation in the orthonormal frame $\{\pr_r,r^{-1}e_1,\ldots,r^{-1}e_{k-1}\}$ gives
\begin{align*}
 |\nabla f|^2
 &=(\pr_ra)^2+|\pr_rF|^2
 +\frac1{r^2}\left(
 |\nabla_Sa-F|^2+|DF|^2
 +2a\dv_SF+(k-1)a^2
 \right),\\
 |\nabla(x\cdot f)|^2
 &=\bigl(\pr_r(ra)\bigr)^2+|\nabla_Sa|^2.
\end{align*}
Using integration by parts on the sphere and the Bochner--Weitzenb\"ock formula,
\[
 \int_{S^{k-1}}|DF|^2\dd\omega
 =\int_{S^{k-1}}\left((\dv_SF)^2+|dF^\flat|^2-(k-2)|F|^2\right)\dd\omega,
\]
$\mathcal Q[f]$ in \eqref{eq:hardy-sharp} equals
\begin{align}
\mathcal Q[f]
&=\int_0^\infty\!\!\int_{S^{k-1}}
 \Bigl\{
 hr^{k-1}\bigl((\pr_ra)^2+(\pr_r(ra))^2+|\pr_rF|^2\bigr)
 \notag\\
&\quad+h^3r^{k-3}|\nabla_Sa|^2
 +\Bigl((k-1)hr^{k-3}-c_k\frac{r^{k-1}}h\Bigr)a^2
 \notag\\
&\quad+\Bigl((3-k)hr^{k-3}-c_k\frac{r^{k-1}}h\Bigr)|F|^2
 +4hr^{k-3}a\dv_SF
 \notag\\
&\quad+hr^{k-3}\bigl((\dv_SF)^2+|dF^\flat|^2\bigr)
 \Bigr\}\dd\omega\dd r.
\label{eq:polar-energy}
\end{align}
Expanding $\pr_r(ra)=a+r\pr_ra$ and integrating its mixed term by parts gives
\begin{align}
&\int_0^\infty hr^{k-1}\bigl((\pr_ra)^2+(\pr_r(ra))^2\bigr)\dd r
\notag\\
&\qquad=\int_0^\infty
 \Bigl[h^3r^{k-1}(\pr_ra)^2
 -\Bigl((k-1)hr^{k-1}+\frac{r^{k+1}}h\Bigr)a^2\Bigr]\dd r.
\label{eq:remove-ra}
\end{align}
The remaining radial derivatives admit the exact square completions
\begin{align}
 \int_0^\infty h^3r^{k-1}(\pr_ra)^2\dd r
 &=\int_0^\infty\Biggl[
 h^3r^{k-1}\Bigl(\pr_ra+\frac{(k+1)r}{2h^2}a\Bigr)^2
 \notag\\[-2mm]
 &\qquad+\Bigl(
 \frac{k+1}{2}(hr^k)'
 -\frac{(k+1)^2}{4}\frac{r^{k+1}}h
 \Bigr)a^2\Biggr]\dd r,
 \label{eq:square-a}\\
 \int_0^\infty hr^{k-1}|\pr_rF|^2\dd r
 &=\int_0^\infty\Biggl[
 hr^{k-1}\Bigl|\pr_rF+\frac{(k-1)r}{2h^2}F\Bigr|^2
 \notag\\[-2mm]
 &\qquad+\Bigl(
 \frac{k-1}{2}\Bigl(\frac{r^k}{h}\Bigr)'
 -\frac{(k-1)^2}{4}\frac{r^{k+1}}{h^3}
 \Bigr)|F|^2\Biggr]\dd r.
 \label{eq:square-F}
\end{align}
Substituting \eqref{eq:remove-ra} into \eqref{eq:polar-energy}, then applying \eqref{eq:square-a} and \eqref{eq:square-F}, discarding the two nonnegative squares, and using $r^2=h^2-1$, we obtain
\begin{align}
 \mathcal Q[f]\ge
 \int_0^\infty\!\!\int_{S^{k-1}}r^{k-3}\Bigl\{
 &h^3|\nabla_Sa|^2+\alpha a^2+4ha\dv_SF
 \notag\\
 &+h\bigl((\dv_SF)^2+|dF^\flat|^2\bigr)-C|F|^2
 \Bigr\}\dd\omega\dd r,
 \label{eq:remainder}
\end{align}
where
\begin{align}
 \alpha=\alpha(r)
 &=\frac{4(k-1)+(k^2+4k-9)r^2+(k-1)^2r^4}{4h},
 \label{eq:alpha}\\
 C=C(r)
 &=(k-1)h-\frac{k^2+7}{4h}+\frac{k^2-1}{4h^3}.
 \label{eq:C}
\end{align}
For $k\ge3$, all coefficients in the numerator of \eqref{eq:alpha} are positive, so $\alpha>0$.

For each fixed $r$, write $a=\bar a+\mathring a$, where $\bar a$ is the spherical mean.  Since $\int_{S^{k-1}}\dv_SF\dd\omega=0$, only $\mathring a$ occurs in the mixed term.  Since $|\nabla_Sa|=|\nabla_S\mathring a|$, Lemma~\ref{lem:sphere} gives
\begin{equation}\label{eq:a-gap}
 \int_{S^{k-1}}h^3|\nabla_Sa|^2\dd\omega
 \ge(k-1)h^3\int_{S^{k-1}}\mathring a^2\dd\omega.
\end{equation}

For the term $\int_{S^{k-1}}\left(h\bigl((\dv_SF)^2+|dF^\flat|^2\bigr)-C|F|^2 \right)d\omega$ on the right-hand side of \eqref{eq:remainder}, direct calculation gives
\begin{equation}\label{eq:sign-C}
 4h^3C=4(k-1)r^4-(k-3)(k-5)r^2+4(k-3).
\end{equation}
The minimum of the right-hand side of \eqref{eq:sign-C} over $r^2\ge0$ is
\[
 \begin{cases}
 4(k-3),&3\le k\le5,\\[1mm]
 \dfrac{k-3}{16(k-1)}
 \left[64(k-1)-(k-3)(k-5)^2\right],&k\ge6.
 \end{cases}
\]
Moreover, $(k-3)(k-5)^2-64(k-1)=k^3-13k^2-9k-11$, whose derivative is $(3k+1)(k-9)$. Direct evaluation gives a negative value for $6\le k\le13$, with the value $-128$ at $k=13$, and the value $59$ at $k=14$. Therefore
\begin{equation}\label{eq:C-nonneg}
 C(r)\ge0\quad\text{for every $r$ when }3\le k\le13,
 \text{ and }
 \{r:C(r)<0\}\ne\varnothing\quad\text{only when }k\ge14.
\end{equation}

If $C(r)\ge0$, multiplying \eqref{eq:oneform-gap} by $\frac{C}{k-1}\ge0$ and adding $h\int_{S^{k-1}}\bigl((\dv_SF)^2+|dF^\flat|^2\bigr)\dd\omega$ gives
\begin{align*}
&\int_{S^{k-1}}\Bigl[
 h\bigl((\dv_SF)^2+|dF^\flat|^2\bigr)-C|F|^2
 \Bigr]\dd\omega\\
&\qquad\ge
 \Bigl(h-\frac{C}{k-1}\Bigr)
 \int_{S^{k-1}}\bigl((\dv_SF)^2+|dF^\flat|^2\bigr)\dd\omega\\
 &\qquad=\frac{8+(k^2+7)r^2}{4(k-1)h^3}
 \int_{S^{k-1}}\bigl((\dv_SF)^2+|dF^\flat|^2\bigr)\dd\omega,
\end{align*}
the last equality by substituting $h^2=1+r^2$ in $h-\frac{C}{k-1}$. If $C(r)<0$, the term $-C|F|^2$ is nonnegative and may be discarded.  Hence, at every $r$,
\begin{align}
&\int_{S^{k-1}}\Bigl[
 h\bigl((\dv_SF)^2+|dF^\flat|^2\bigr)-C|F|^2
 \Bigr]\dd\omega
 \notag\\
&\qquad\ge\gamma
 \int_{S^{k-1}}\bigl((\dv_SF)^2+|dF^\flat|^2\bigr)\dd\omega,
 \label{eq:F-absorb}
\end{align}
where
\begin{equation}\label{eq:gamma}
 \gamma=
 \begin{cases}
 \dfrac{8+(k^2+7)r^2}{4(k-1)h^3},&C\ge0,\\[2mm]
 h,&C<0,
 \end{cases}
\end{equation}
and $\gamma>0$ in both cases.

Set $\beta=(k-1)h^3+\alpha$. Combining \eqref{eq:remainder}, \eqref{eq:a-gap}, and \eqref{eq:F-absorb}, we find
\begin{align}
 \mathcal Q[f]\ge\int_0^\infty r^{k-3}\Biggl\{
 &\alpha\int_{S^{k-1}}\bar a^2\dd\omega
 \notag\\
 &+\int_{S^{k-1}}\bigl(
 \beta\mathring a^2+\gamma(\dv_SF)^2
 +4h\mathring a\dv_SF+\gamma|dF^\flat|^2
 \bigr)\dd\omega\Biggr\}\dd r.
\label{eq:final-lower}
\end{align}
The only remaining mixed term is handled by
\begin{equation}\label{eq:last-square}
 \beta\mathring a^2+\gamma(\dv_SF)^2+4h\mathring a\dv_SF
 =\gamma\Bigl(\dv_SF+\frac{2h}{\gamma}\mathring a\Bigr)^2
 +\frac{\beta\gamma-4h^2}{\gamma}\mathring a^2.
\end{equation}
If $C\ge0$, substituting the expression in \eqref{eq:gamma} for $C\geq 0$ into \eqref{eq:last-square} gives
\begin{align*}
 \beta\gamma-4h^2
 =\frac{r^2}{16(k-1)h^4}\Bigl[&8k(k^2-5)
 +(k^4+12k^3-2k^2-92k+49)r^2
 \notag\\
 &+(k-1)(k^3+3k^2+7k-43)r^4\Bigr].
\end{align*}
All three coefficients in the bracket are positive for every integer $k\ge3$: at $k=3$ they equal $96$, $160$, and $64$, and each is increasing in $k$. Thus $\beta\gamma-4h^2\ge0$ whenever $C\ge0$. If $C<0$, then $k\ge14$ by \eqref{eq:C-nonneg}, and, using $\alpha>0$ and $h\ge1$,
\[
 \beta\gamma=\beta h\ge(k-1)h^4\ge13h^4\ge4h^2.
\]
Thus $\beta\gamma-4h^2\ge0$ in both cases, and every term in \eqref{eq:final-lower} is nonnegative. This proves \eqref{eq:hardy-sharp}.
\end{proof}

\begin{remark}\label{rem:ck-optimal}
The constant $c_k$ in \eqref{eq:hardy-sharp} can be shown to be optimal.  A key ingredient is that $F_1=\nabla_S\omega_1=e_1-\langle\omega,e_1\rangle\,\omega$ attains equality in \eqref{eq:oneform-gap}. 
\end{remark}

To apply Proposition~\ref{prop:sharp-hardy} to $q_{\bla}$ in \eqref{eq:L^2-functional}, observe that
\begin{equation}\label{eq:ck-minus-3}
 c_k-3=\frac{(k-3)(k+1)}4\ge0
 \qquad\text{for }k\ge3.
\end{equation}

\begin{theorem}\label{thm:stable-half}
$\Hk$ is stable for every $k\ge3$.
\end{theorem}

\begin{proof}
Take $\bla=(1,\ldots,1)$ in \eqref{eq:L^2-functional}.  Then $\Lambda$
is the identity and $|\Lambda f|^2=|f|^2$ for any $f\in C_c^\infty(\R^k,\R^k)$, so
\[
 q_{\bla}[f]=\int_{\R^k}h\left(
 \sum_{j=1}^k|\nabla f^j|^2+|\nabla(u\cdot f)|^2
 -\frac{3|f|^2}{h^2}
 \right)du .
\]
Proposition~\ref{prop:sharp-hardy} and \eqref{eq:ck-minus-3} give $q_{\bla}[f]\ge (c_k-3)\int_{\R^k}\frac{|f|^2}{h}\,du\ge0$. Stability now follows from Proposition~\ref{prop:stability-reduction}.
\end{proof}

 This proves the stable half of Theorem \ref{thm:stability}. The instability of $H_2$ is proved in Section~\ref{sec:unstable}; the instability of the classical helicoid $H_1$ follows from \cite{dCP,FCS}.

\section{Calibration of \texorpdfstring{$H_k$}{Hk} for even \texorpdfstring{$k\geq6$}{k>=6}}\label{sec:calibration-Hk-even}

Throughout this section, $k\ge6$ is even. Let
\[
 H_k=\left\{(\cos s\,u,\sin s\,u,s):u\in\R^k,\ s\in\R\right\}\subset\R^{2k+1}
\] 
be the $(k+1)$-dimensional helicoid. The projection of $H_k$ onto $\R^{2k}$ is precisely $C(2, k, 1)$, the cone of $2\times k$ matrices of rank at most $1$. Kerckhove and Lawlor \cite{KerckhoveLawlor} studied the area-minimizing property of these minimal cones in $\R^{2k}$ and showed that $C(2, k, 1)$ is area-minimizing if $k\geq 4$; see also \cite{LawlorCriterion}. The construction of our calibration is inspired by their work.

\subsection{Coordinates adapted to the helicoid and the induced geometry}\label{sec:calibration-coordinates}
Write
$
\mathbb R^{2k}=\mathbb R^k\times\mathbb R^k,
$
and denote a point by \((x,y)\), where \(x,y\in\mathbb R^k\).

Consider the subset $Z_0\subset \mathbb R^{2k}$:
\begin{equation}\label{eq:Z_0}
Z_0 = \{ (x,y) \in \R^{2k} : |x|=|y|, \, \langle x,y\rangle = 0 \} ,
\end{equation} 
and define $\sigma_1\geq 0$, $\sigma_2\geq 0$ by
\[
\sigma_1^2=
\frac{|x|^2+|y|^2}{2}
+
\frac12
\sqrt{\bigl(|x|^2-|y|^2\bigr)^2+4\langle x, y\rangle^2}
\]
and
\[
\sigma_2^2=
\frac{|x|^2+|y|^2}{2}
-\frac12
\sqrt{\bigl(|x|^2-|y|^2\bigr)^2+4\langle x, y\rangle^2}.
\]

Then $\sigma_1$ and $\sigma_2$ are the singular values of the $2\times k$ matrix $\begin{pmatrix}
x\\y
\end{pmatrix}=\begin{pmatrix}
x_1 & x_2 & \cdots & x_k\\
y_1 & y_2 & \cdots & y_k
\end{pmatrix}$. Moreover, $Z_0=\{(x, y)\in \mathbb R^{2k}:\sigma_1=\sigma_2\}$ and therefore $\sigma_1>\sigma_2$ on 
$\R^{2k}\setminus Z_0$.

For any $(x, y) \in \R^{2k}\setminus Z_0$, by singular value decomposition there exist $s\in \mathbb R$, and orthonormal vectors $w, v\in \mathbb R^k$ such that 
\begin{equation}\label{eq:SVD}
 \begin{pmatrix}
x\\y
\end{pmatrix}
 =\begin{pmatrix}\cos s&-\sin s\\ \sin s&\cos s\end{pmatrix}
 \begin{pmatrix}\sigma_1 &0\\ 0&\sigma_2\end{pmatrix}
 \begin{pmatrix}
w\\v
\end{pmatrix}.
\end{equation}
In particular, $w$ is associated with the larger singular value $\sigma_1$.
 
We also introduce $\rho>0$ and $0\leq t<\frac{\pi}{4}$ such that 
\begin{equation}\label{eq:adapted-singular-values}
 \sigma_1=\rho\cos t,\qquad \sigma_2=\rho\sin t
\end{equation} 
in $\R^{2k}\setminus Z_0$. Together, $(\rho,t,s,w,v)$ form a coordinate system on $\R^{2k}\setminus Z_0$ with
\[
\rho>0,\qquad 0\leq t<\frac{\pi}{4},\qquad s\in\mathbb R,\qquad w, v\in \mathbb R^k,\qquad |w|=|v|=1, \langle w, v\rangle=0.
\] 

Since we choose
$
\begin{pmatrix}
\cos s & -\sin s\\
\sin s & \cos s
\end{pmatrix}\in SO(2),
$
the only remaining ambiguity, apart from the standard periodicity \(s\sim s+2\pi\), is
$
(s,w,v)\sim (s+\pi,-w,-v).$

For $(x,y)\in\R^{2k}\setminus Z_0$, \eqref{eq:adapted-singular-values} gives
\begin{equation}\label{eq:rho-t-in-xy}
 \rho=\sqrt{|x|^2+|y|^2},
 \qquad
 \sin(2t)=\frac{2\sqrt{|x|^2|y|^2-\langle x, y\rangle^2}}
 {|x|^2+|y|^2}.
\end{equation}
Adding $z\in\R$ gives adapted coordinates $(\rho,t,s,w,v,z)$ on $\R^{2k+1}\setminus Z$, where $Z=Z_0 \times \mathbb{R}$, in which
\begin{equation}\label{eq:helicoid-adapted-coordinates}
 H_k=\{t=0,\ z=s\}.
\end{equation}

Write $d\tau=\langle dw,v\rangle$. Differentiating \eqref{eq:SVD} and computing the Euclidean metric on $\R^{2k}\setminus Z_0$, we obtain
\begin{align*}
 \left|dx\right|^2+\left|dy\right|^2&=\left|\cos t\,w\,d\rho-\rho\sin t\,w\,dt
 -\rho\sin t\,v\,ds+\rho\cos t\,dw\right|^2\\
 &\quad+\left|\sin t\,v\,d\rho+\rho\cos t\,v\,dt
 +\rho\cos t\,w\,ds+\rho\sin t\,dv\right|^2.
\end{align*}
The relations $|w|=|v|=1$ and $\langle w, v\rangle=0$ give
\[
 dw=v\,d\tau+(dw-v\,d\tau),
 \qquad
 dv=-w\,d\tau+(dv+w\,d\tau).
\]
Hence the Euclidean metric on $\R^{2k+1}\setminus Z$ is
\begin{equation}\label{eq:adapted-metric}
\begin{split}
 g={}&d\rho^2+\rho^2dt^2
 +\rho^2\left(ds^2-2\sin(2t)\,ds\,d\tau+d\tau^2\right)\\
 &+\rho^2\cos^2t\,|dw-v\,d\tau|^2
 +\rho^2\sin^2t\,|dv+w\,d\tau|^2+dz^2.
\end{split}
\end{equation}
Restricting \eqref{eq:adapted-metric} to $H_k$ as described in \eqref{eq:helicoid-adapted-coordinates} gives
\begin{equation}\label{eq:helicoid-polar-metric}
 g_{H_k}=d\rho^2+(1+\rho^2)ds^2+\rho^2|dw|^2.
\end{equation}

Let $\Theta$ denote the standard volume form of the metric $|dw|^2$ on $S^{k-1}$,
\begin{equation}\label{eq:Theta}
\Theta
=
\sum_{i=1}^k
(-1)^{i-1} w_i\,
dw_1\wedge\cdots\wedge \widehat{dw_i}\wedge\cdots\wedge dw_k .
\end{equation}
When $k$ is even, $\Theta$ is invariant under $(s,w,v)\sim(s+\pi,-w,-v)$ and hence descends to $\mathbb{R}^{2k+1}\setminus Z$.

By \eqref{eq:helicoid-polar-metric}, we orient the $(k+1)$-dimensional helicoid $H_k$ by
\begin{equation}\label{eq:helicoid-polar-volume}
 d\mu_{H_k}=\rho^{k-1}\sqrt{1+\rho^2}\,
 \Theta\wedge d\rho\wedge ds.
\end{equation}

Choose a local orthonormal frame $e_1,\ldots,e_{k-2}$ of $\{w,v\}^{\perp}$ such that $(v,e_1,\ldots,e_{k-2})$ is positively oriented in $T_wS^{k-1}$, and set
\[
 \theta^a=\langle dw-v\,d\tau,e_a\rangle=\langle dw,e_a\rangle,
 \qquad 1\leq a\leq k-2.
\]
Then
\begin{equation}\label{eq:angular-volume-form}
 \Theta=d\tau\wedge\theta^1\wedge\cdots\wedge\theta^{k-2}
\end{equation} and $|dw-v\, d\tau|^2=\sum_{a=1}^{k-2} (\theta^a)^2$.

In the rest of this subsection, we consider $\Theta$ and $\Theta \wedge ds$ as a $(k-1)$-form and a $k$-form on $\mathbb{R}^{2k+1}\setminus Z$ and compute their comasses with respect to $g$ in \eqref{eq:adapted-metric}.

The relevant part of the inverse of \eqref{eq:adapted-metric} is
\[
 \langle ds,ds\rangle=\langle d\tau,d\tau\rangle
 =\frac{1}{\rho^2\cos^2(2t)},
 \qquad
 \langle ds,d\tau\rangle
 =\frac{\sin(2t)}{\rho^2\cos^2(2t)}.
\]
Moreover, $\theta^1,\ldots,\theta^{k-2}$ are mutually orthogonal and orthogonal to $ds$ and $d\tau$, with
\[
 \|\theta^a\|^2=\frac{1}{\rho^2\cos^2t}.
\]
By \eqref{eq:angular-volume-form}, $\Theta$ is decomposable, and its comass equals its pointwise norm.
Therefore,
\[
 \|\Theta\|_*^2
 =\frac{1}{\rho^2\cos^2(2t)}
 \left(\frac{1}{\rho^2\cos^2t}\right)^{k-2}
 =\frac{1}{\rho^{2k-2}J(t)^2}
\]
where
\begin{equation}\label{eq:J-definition}
 J(t)=(\cos t)^{k-2}\cos(2t).
\end{equation}

Similarly,
\[
 \|d\tau\wedge ds\|^2
 =\det\begin{pmatrix}
 \langle d\tau,d\tau\rangle&\langle d\tau,ds\rangle\\
 \langle ds,d\tau\rangle&\langle ds,ds\rangle
 \end{pmatrix}
 =\frac{1}{\rho^4\cos^2(2t)}.
\]
Consequently, using \eqref{eq:J-definition},
\begin{equation}\label{eq:angular-comass}
 \|\Theta\|_*=\frac{1}{\rho^{k-1}J(t)},
 \qquad
 \|\Theta\wedge ds\|_*=\frac{1}{\rho^kJ(t)}.
\end{equation} 

\subsection{The calibration}\label{sec:calibration-form}

We construct the calibration $\xi$ for $H_k$ on $\mathbb{R}^{2k+1}\setminus Z$ in this subsection.

\begin{proposition}\label{prop:calibration}
Let
\[
\varphi(t)=\cos(2t), \qquad 0\leq t<\frac{\pi}{4},
\]
and set \(r=\rho\varphi(t)\). On \(\mathbb R^{2k+1}\setminus Z\), with coordinates $
(\rho,t,s,w,v,z)$
define
\begin{equation}\label{eq:xi-decomposable}
\xi
=
\Theta\wedge dr\wedge
\left(
\frac{r^{k+1}}{\sqrt{1+r^2}}\,ds
+
\frac{r^{k-1}}{\sqrt{1+r^2}}\,dz
\right).
\end{equation}
Then \(\xi\) has the following properties on \(\mathbb R^{2k+1}\setminus Z\):
\begin{enumerate}
\item Its comass satisfies
\begin{equation}\label{eq:comass-bound}
\|\xi\|_*\leq 1.
\end{equation}

\item The form \(\xi\) satisfies
\begin{equation}\label{eq:calibration-regularity}
\xi\in C^\infty(\mathbb R^{2k+1}\setminus Z)
\cap L^\infty_{\mathrm{loc}}(\mathbb R^{2k+1}),
\qquad
d\xi=0
\quad\text{on }\mathbb R^{2k+1}\setminus Z.
\end{equation}

\item The restriction of \(\xi\) to \(H_k\setminus Z\) is the induced volume form:
\begin{equation}\label{eq:calibration-restriction}
\left.\xi\right|_{H_k\setminus Z}
=
d\mu_{H_k}.
\end{equation}
\end{enumerate}

\end{proposition}

\begin{proof}

We first estimate the comass of $\xi$. Differentiating $r=\rho \varphi(t) $ gives
\[
 dr=\varphi\,d\rho+\rho\varphi'\,dt.
\]
By \eqref{eq:adapted-metric}, $d\rho\perp dt$, $|d\rho|=1$, and $|dt|=\rho^{-1}$; moreover, $dr$ is orthogonal to all the one-form factors in $\Theta$, $ds$, and $dz$. Hence
\begin{equation*}
 |dr|^2=\varphi^2+(\varphi')^2.
\end{equation*}
The form $\xi$ is decomposable, so its comass equals its pointwise norm. Since $dz$ is a unit one-form orthogonal to $ds$ and to every factor in $\Theta$, \eqref{eq:angular-comass} gives
\begin{align*}
 \|\xi\|_*^2
 &=\left(\varphi^2+(\varphi')^2\right)
 \left(\frac{r^{2k+2}}
 {1+r^2}\|\Theta\wedge ds\|_*^2
 +\frac{r^{2k-2}}
 {1+r^2}\|\Theta\wedge dz\|_*^2\right)\\
 &=\frac{\varphi^2+(\varphi')^2}{\rho^{2k}J(t)^2}
 \left(\frac{r^{2k+2}}
 {1+r^2}+\rho^2\frac{r^{2k-2}}
 {1+r^2}\right).
\end{align*}

Substituting $r=\rho\varphi$ we obtain
\begin{equation}\label{eq:calibration-comass}
 \|\xi\|_*^2
 =\frac{\varphi^{2k-2}\left(\varphi^2+(\varphi')^2\right)}{J(t)^2}
 \frac{1+\rho^2\varphi^4}{1+\rho^2\varphi^2}.
\end{equation}
Since $0\leq\varphi\leq1$, the last factor in \eqref{eq:calibration-comass} is at most one. The comass bound \eqref{eq:comass-bound} now follows from Lemma~\ref{lem:trig-inequality} at the end of this subsection.

To prove (2), let $A,B:[0,\infty)\to\R$ satisfy 
 \begin{equation}\label{eq:AB-definition}
 A'(r)=\frac{r^{k+1}}{\sqrt{1+r^2}},
 \qquad
 B'(r)=\frac{r^{k-1}}{\sqrt{1+r^2}}
\end{equation}
 with $A(0)=B(0)=0$. On $\R^{2k+1}\setminus Z$, define
\begin{equation}\label{eq:Psi-definition}
 \Psi=(-1)^{k-1}\left(A(r)\Theta\wedge ds
 +B(r)\Theta\wedge dz\right).
\end{equation} 
Since $d\Theta=0$, $d\Psi=\xi$ and thus $\xi$ is closed on $\R^{2k+1}\setminus Z$.

We then control $\Psi$ and $\xi$ near $Z$. By \eqref{eq:adapted-singular-values} and \eqref{eq:rho-t-in-xy},
\[
\cos(2t)=\frac{\sqrt{(|x|^2-|y|^2)^2+4\langle x, y\rangle^2}}{|x|^2+|y|^2},
\]
and hence
\[
r=\left(\frac{(|x|^2-|y|^2)^2+4\langle x, y\rangle^2}{|x|^2+|y|^2}\right)^{1/2}.
\]
The function $r$ is smooth on $\R^{2k+1}\setminus Z$. Since $\Theta$, $ds$, and $dz$ are also smooth there, \eqref{eq:AB-definition} and \eqref{eq:Psi-definition} give $\Psi,\xi\in C^\infty(\R^{2k+1}\setminus Z)$ and $d\Psi=\xi$ on $\R^{2k+1}\setminus Z$. Directly from \eqref{eq:AB-definition},
\[
0\leq A(r)\leq\frac{r^{k+2}}{k+2},\qquad 0\leq B(r)\leq\frac{r^k}{k}.
\]
Using \eqref{eq:angular-comass} and $r=\rho\varphi$ gives
\[
\|\Psi\|_*
\leq A(r)\|\Theta\wedge ds\|_*+B(r)\|\Theta\wedge dz\|_*
\leq\frac{\rho^2\varphi^{k+2}}{(k+2)J(t)}
+\frac{\rho\varphi^k}{kJ(t)}.
\]
Moreover,
\[
\frac{\varphi^k}{J(t)}=\frac{(\cos(2t))^{k-1}}{(\cos t)^{k-2}},\qquad \frac{\varphi^{k+2}}{J(t)}=\frac{(\cos(2t))^{k+1}}{(\cos t)^{k-2}}.
\]
Both ratios are bounded on $[0,\pi/4]$. Defining $\Psi=\xi=0$ on $Z$, which does not change their $L^\infty_{\mathrm{loc}}$ classes, we obtain \eqref{eq:calibration-regularity}, that is,
\begin{equation*}
 \Psi,\xi\in C^\infty(\R^{2k+1}\setminus Z)
 \cap L^\infty_{\mathrm{loc}}(\R^{2k+1}),
 \qquad
 d\Psi=\xi\quad\text{on }\R^{2k+1}\setminus Z.
\end{equation*}
The local boundedness of $\xi$ follows from \eqref{eq:comass-bound}.

Finally, on $H_k\setminus Z$, one has $t=0$, $z=s$, $r=\rho$, $dr=d\rho$, and $dz=ds$. Hence, \eqref{eq:xi-decomposable} and \eqref{eq:helicoid-polar-volume} give
\begin{equation*}
 \left.\xi\right|_{H_k\setminus Z}
 =\rho^{k-1}\sqrt{1+\rho^2}\,
 \Theta\wedge d\rho\wedge ds
 =d\mu_{H_k}.
\end{equation*}

\end{proof}

\begin{lemma}\label{lem:trig-inequality}
For $k\geq6$, the function $\varphi(t)=\cos(2t)$ satisfies $0\leq\varphi\leq1$, $\varphi(0)=1$, $\varphi(\pi/4)=0$, and
\begin{equation}\label{eq:trig-inequality}
 \varphi^{2k-2}\left(\varphi^2+(\varphi')^2\right)\leq J(t)^2
 \qquad\text{on }\left[0,\frac{\pi}{4}\right).
\end{equation}
\end{lemma}

\begin{proof}
Since $\varphi'=-2\sin(2t)$, $\varphi^2+(\varphi')^2=1+3\sin^2(2t)$. Set $\mathsf{t}=\tan^2t\in[0,1)$. Using
\[
 \cos(2t)=\frac{1-\mathsf{t}}{1+\mathsf{t}},
 \qquad
 \sin^2(2t)=\frac{4\mathsf{t}}{(1+\mathsf{t})^2},
 \qquad
 \cos^2t=\frac1{1+\mathsf{t}},
\]
inequality \eqref{eq:trig-inequality} is equivalent to
\[
 f(\mathsf{t}):=
 \frac{(1-\mathsf{t})^{2k-4}(1+14\mathsf{t}+\mathsf{t}^2)}
 {(1+\mathsf{t})^k}\le1.
\]
Its logarithmic derivative is
\begin{align*}
 \frac{d}{d\mathsf{t}}\log f(\mathsf{t})
 &=-\frac{2k-4}{1-\mathsf{t}}
 +\frac{14+2\mathsf{t}}{1+14\mathsf{t}+\mathsf{t}^2}
 -\frac{k}{1+\mathsf{t}}\\
 &=-(k-6)\frac{3+\mathsf{t}}{1-\mathsf{t}^2}
 +(14+2\mathsf{t})\left(
 \frac1{1+14\mathsf{t}+\mathsf{t}^2}-\frac1{1-\mathsf{t}^2}
 \right)\le0.
\end{align*}
Thus $f$ is non-increasing on $[0,1)$, and since $f(0)=1$, the result follows.
\end{proof}

\subsection{Distributional closedness and area minimization}\label{sec:GMT}

Let $d(\cdot)=\dist(\cdot,Z)$. The extension across $Z$ uses the following codimension-two tubular estimate.

\begin{lemma}\label{lem:singular-set-estimate}
For every $R>0$, there exists $C_{k,R}>0$ such that
\begin{equation*}
 \Vol\bigl(\{d\leq\epsilon\}\cap B_R^{2k+1}(0)\bigr)
 \leq C_{k,R}\epsilon^2,
 \qquad \forall\epsilon>0.
\end{equation*}
\end{lemma}

\begin{proof}
Write $Z = Z_0 \times \R$, where, as recall from \eqref{eq:Z_0},
\[
 Z_0 = \{ (x,y) \in \R^{2k} : |x|=|y|, \, x \cdot y = 0 \},
\]
so that $d((x,y,z), Z) = \dist((x,y), Z_0)$. Set $K = Z_0 \cap S^{2k-1}$. On $S^{2k-1}$,
\[
 K=F^{-1}(0,0), \text{ where }F(x,y) = (|x|^2 - |y|^2, \, 2x \cdot y).
\]
At every point of $K$, the spherical gradients of the components of $F$ are $(2x, -2y)$ and $(2y, 2x)$, which are orthogonal and have norm $2$. Thus, $K \subset S^{2k-1}$ is a smooth compact submanifold of codimension $2$, and consequently
\[
\Vol_{2k-1}\left(\{ \theta \in S^{2k-1} : \dist(\theta, K) \le \delta \}\right) \le \tilde{c}_k\,\delta^2, \qquad \forall \delta > 0.
\]

Let $p = r\theta$ with $r > \epsilon$, $\theta \in S^{2k-1}$ and suppose $\dist(p, Z_0) \le \epsilon$. Choosing a closest point $s\eta \in Z_0$ with $s > 0$ and $\eta \in K$, we have $|r\theta - s\eta| \le \epsilon$ and $|r-s| \le \epsilon$, which yields
\[
\dist(\theta, K) \le |\theta - \eta| \le \frac{|r\theta - s\eta| + |r-s|}{r} \le \frac{2\epsilon}{r}.
\]

Integrating in polar coordinates for $0 < \epsilon \le R$ gives
\begin{align*}
\Vol\bigl(\{d \le \epsilon\} \cap B_R^{2k+1}\bigr)
&\le 2R \, \Vol_{2k}\bigl(\{ \dist(\cdot, Z_0) \le \epsilon \} \cap B_R^{2k}\bigr) \\
&= 2R \left[ \Vol_{2k}\bigl(\{ \dist(\cdot, Z_0) \le \epsilon \} \cap B_\epsilon^{2k}\bigr)\right.\\
&\hspace{4.8em}\left.{}+ \Vol_{2k}\bigl(\{ \dist(\cdot, Z_0) \le \epsilon \} \cap (B_R^{2k} \setminus B_\epsilon^{2k})\bigr) \right] \\
&\le 2R \left[ \Vol_{2k}(B_\epsilon^{2k})\right.\\
&\hspace{4.8em}\left.{}+ \int_\epsilon^R r^{2k-1} \Vol_{2k-1}\left(\left\{ \theta \in S^{2k-1} : \dist(\theta, K) \le \frac{2\epsilon}{r} \right\}\right) dr \right] \\
&\le 2R \left[ \omega_{2k}\epsilon^{2k} + \int_\epsilon^R r^{2k-1} \cdot c_k \left(\frac{2\epsilon}{r}\right)^2 dr \right] \\
&= 2R \left[ \omega_{2k}\epsilon^{2k-2} + \frac{2c_k}{k-1}\left(R^{2k-2} - \epsilon^{2k-2}\right) \right] \epsilon^2 \\
&\le C_{k,R} \epsilon^2.
\end{align*}
For $\epsilon > R$, the estimate follows immediately by bounding the volume by $\Vol(B_R^{2k+1})$.
\end{proof}

We now extend the relation $d\Psi=\xi$ across $Z$. Fix a test $k$-form $\eta \in C_c^\infty(\R^{2k+1}; \Omega^k)$ with $\operatorname{supp} \eta \subset B_R(0)$. Let $\chi_\epsilon \in C^\infty(\R^{2k+1})$ be a cutoff function with $\chi_\epsilon = 0$ on $\{d \le \epsilon\}$, $\chi_\epsilon = 1$ on $\{d \ge 2\epsilon\}$, and $|\nabla \chi_\epsilon| \le \frac{2}{\epsilon}$. By \eqref{eq:calibration-regularity}, $\Psi$ and $\xi$ are smooth on $\operatorname{supp}\chi_\epsilon$ and $d\Psi=\xi$ there. Hence
\[
\int_{\R^{2k+1}} \chi_\epsilon \xi \wedge \eta = (-1)^{k+1} \int_{\R^{2k+1}} \chi_\epsilon \Psi \wedge d\eta + (-1)^{k+1} \int_{\R^{2k+1}} d\chi_\epsilon \wedge \Psi \wedge \eta.
\]
Using Lemma~\ref{lem:singular-set-estimate} and $\Psi \in L^\infty_{\mathrm{loc}}(\R^{2k+1}; \Omega^k)$, the error term satisfies
\[
\left| \int_{\R^{2k+1}} d\chi_\epsilon \wedge \Psi \wedge \eta \right| \le \|\Psi\|_{L^\infty(B_R)} \|\eta\|_{L^\infty} \cdot \frac{2}{\epsilon} \Vol(\{\epsilon \le d \le 2\epsilon\} \cap B_R) \le C_R' \epsilon \to 0
\]
as $\epsilon \to 0$. Since $\Psi, \xi \in L^\infty_{\mathrm{loc}}$, dominated convergence yields
\begin{equation*}
\int_{\R^{2k+1}} \xi \wedge \eta = (-1)^{k+1} \int_{\R^{2k+1}} \Psi \wedge d\eta.
\end{equation*}
Thus $d\Psi=\xi$ distributionally. Consequently, for any $\zeta \in C_c^\infty(\R^{2k+1}; \Omega^{k-1})$,
\[
\int_{\R^{2k+1}} \xi \wedge d\zeta = (-1)^{k+1} \int_{\R^{2k+1}} \Psi \wedge d^2\zeta = 0,
\]
so
\begin{equation}\label{eq:distributional-closedness}
 d\xi=0
\end{equation}
distributionally on $\R^{2k+1}$.

\begin{proof}[Proof of Theorem~\ref{thm:area-minimizing}]
The form $\xi$ is distributionally closed by \eqref{eq:distributional-closedness}, has comass at most one by \eqref{eq:comass-bound}, and restricts to the oriented volume form of $H_k\setminus Z$ by \eqref{eq:calibration-restriction}. Since $H_k\cap Z$ has zero $(k+1)$-dimensional measure, $\xi$ calibrates $H_k$.

Let $M'$ be a smooth oriented competitor homologous to $H_k$ that agrees with $H_k$ outside a compact set. Choose compact portions $P\subset H_k$ and $P'\subset M'$ containing the region where they differ, and a compact oriented $(k+2)$-chain $S$ such that $P-P'=\partial S$. Fix a nonnegative function $\vartheta\in C_c^\infty(B_1^{2k+1}(0))$ satisfying $\int_{\R^{2k+1}}\vartheta=1$. For $\epsilon>0$, define
\[
 \vartheta_\epsilon(p)=\epsilon^{-(2k+1)}
 \vartheta\left(\frac{p}{\epsilon}\right),
 \qquad
 \xi_\epsilon(p)
 =(\vartheta_\epsilon*\xi)(p)
 =\int_{\R^{2k+1}}\vartheta_\epsilon(q)\xi(p-q)\,dq.
\]
Then $\operatorname{supp}\vartheta_\epsilon\subset B_\epsilon^{2k+1}(0)$. Since $d\xi=0$ distributionally, $d\xi_\epsilon =d(\vartheta_\epsilon*\xi) =\vartheta_\epsilon*d\xi =0$.
Moreover, the comass bound is preserved under convolution
\[
 \|\xi_\epsilon(p)\|_*
 \leq\int_{\R^{2k+1}}
 \vartheta_\epsilon(q)\|\xi(p-q)\|_*\,dq
 \leq\int_{\R^{2k+1}}\vartheta_\epsilon(q)\,dq
 =1.
\]
Since $\xi_\epsilon$ is smooth and closed, Stokes' theorem gives
\[
 0=\int_Sd\xi_\epsilon
 =\int_{\partial S}\xi_\epsilon
 =\int_P\xi_\epsilon-\int_{P'}\xi_\epsilon.
\]
On $H_k\setminus Z$, the forms $\xi_\epsilon$ converge locally uniformly to $\xi$. Therefore, by dominated convergence and \eqref{eq:calibration-restriction},
\[
 \Vol(P)
 =\lim_{\epsilon\to0}\int_P\xi_\epsilon
 =\lim_{\epsilon\to0}\int_{P'}\xi_\epsilon
 \leq\Vol(P').
\]
Therefore $H_k$ is area-minimizing.
\end{proof}
\section{Calibration for \texorpdfstring{$H_4$}{H4}}\label{sec:calibration-H4}

We identify $\R^8=\mathbb{H}\oplus\mathbb{H}$ and write a point of $\R^9$ as $(x,y,z)$, where $x,y\in\mathbb{H}$. Writing $x=x_0+x_1i+x_2j+x_3k$ and $y=y_0+y_1i+y_2j+y_3k$, the three K\"ahler forms induced by simultaneous left multiplication by $i,j,k$ on the two quaternionic factors are
\begin{align*}
    \omega_1&=dx_0\wedge dx_1+dx_2\wedge dx_3+dy_0\wedge dy_1+dy_2\wedge dy_3,\\
    \omega_2&=dx_0\wedge dx_2-dx_1\wedge dx_3+dy_0\wedge dy_2-dy_1\wedge dy_3,\\
    \omega_3&=dx_0\wedge dx_3+dx_1\wedge dx_2+dy_0\wedge dy_3+dy_1\wedge dy_2.
\end{align*}
The $4$-form introduced by Kraines \cite{KrainesQuaternionic} is
\begin{equation}\label{eq:Kraines-4-form}
    \kappa=\frac16(\omega_1^2+\omega_2^2+\omega_3^2),
\end{equation}
normalized to have comass one. Set $\rho=\sqrt{|x|^2+|y|^2}$ and $h=\sqrt{1+\rho^2}$. On the same $\R^8$, define the $2$-form
\[
    \omega=\sum_{a=0}^3dx_a\wedge dy_a.
\]
Equivalently, $\omega(U,V)=\langle(-U_2,U_1),V\rangle$ for $U=(U_1,U_2)$ and $V\in\mathbb{H}\oplus\mathbb{H}$. It is the K\"ahler form of the orthogonal complex structure $(x,y)\mapsto(-y,x)$. Since $\rho\pr_\rho=(x,y)$, one has $\iota_{\rho\pr_\rho}\omega=(-y,x)^\flat$, so this contraction is precisely the horizontal part of the screw direction. All these forms are defined on $\R^8$ and are pulled back to $\R^9=\R^8\times\R_z$ without further notation.

Recall from \eqref{eq:helicoid-parametrization} that
\[
    H_4=\{(\cos s\,u,\sin s\,u,s):u\in\mathbb{H},\ s\in\R\}.
\]
The restriction of $\rho$ to $H_4$ is $|u|$. Writing $u=u_0+u_1i+u_2j+u_3k$,  the induced metric and volume form are
\begin{align}\label{eq:h4-induced-geometry} \begin{split}
    g_{H_4} 
    &=\sum_{a=0}^3(du^a)^2+(1+|u|^2)ds^2, \\
    d\mu_{H_4}&=\sqrt{1+|u|^2}\,du^0\wedge du^1\wedge du^2\wedge du^3\wedge ds.
\end{split} \end{align}
At $(x,y,z)=(\cos z\,u,\sin z\,u,z)\in H_4$,
\begin{equation}\label{eq:h4-tangent-splitting}
    T_{(x,y,z)}H_4=\{(\cos z\,v,\sin z\,v,0):v\in\mathbb{H}\}\oplus\operatorname{span}\{(-y,x)+\pr_z\}.
\end{equation}
We orient $H_4$ by the parametrization $(u,s)\mapsto(\cos s\,u,\sin s\,u,s)$. The first summand of \eqref{eq:h4-tangent-splitting} is the quaternionic ruling at height $z$. Denote its oriented unit simple $4$-vector by $P_z$, where
\begin{align}\label{eq:h4-Pz}
    P_z &= (\cos z,\sin z,0)\wedge(\cos z\,i,\sin z\,i,0)\wedge(\cos z\,j,\sin z\,j,0)\wedge(\cos z\,k,\sin z\,k,0).
\end{align}
Direct substitution of \eqref{eq:h4-Pz} into \eqref{eq:Kraines-4-form} gives $\kappa(P_z)=1$. By \eqref{eq:h4-induced-geometry}, the splitting in \eqref{eq:h4-tangent-splitting} is orthogonal, $|P_z|=1$, and
\begin{equation}\label{eq:h4-tangent-volume}
 d\mu_{H_4}(P_z,(-y,x)+\pr_z)=|P_z\wedge((-y,x)+\pr_z)|=\sqrt{1+\rho^2}=h.
\end{equation}

Set
\begin{equation}\label{eq:h4-A-B}
 A(\rho)=\frac{3h^2+9h+8}{15(h+1)^3},\qquad B(\rho)=\frac{h+2}{3(h+1)^2}
\end{equation}
and define
\begin{equation}\label{eq:h4-potential}
 \Psi=\iota_{\rho\pr_\rho}\kappa\wedge\left(A(\rho)\,\iota_{\rho\pr_\rho}\omega + B(\rho)\,dz\right),\qquad \xi=d\Psi.
\end{equation}
In the following proposition, $\xi$ will be shown to be a calibration form on $\R^9$ that calibrates $H_4$.

\begin{remark}\label{rem:ansatz}
The motivation for the ansatz \eqref{eq:h4-potential} is as follows. Since
\[
d\left(\frac14\iota_{\rho\pr_\rho}\kappa\right)=\kappa,
\qquad
\kappa(P_z)=1,
\]
the form $\frac14\iota_{\rho\pr_\rho}\kappa$ provides a natural $3$-form whose exterior derivative agrees with the volume form along each quaternionic ruling. The vertical and horizontal parts of the remaining tangent direction $(-y,x)+\pr_z$ are detected by $dz$ and $\iota_{\rho\pr_\rho}\omega=(-y,x)^\flat$, respectively. This leads to the two terms in \eqref{eq:h4-potential}. Comparing with the form $\Psi$ defined in \eqref{eq:Psi-definition} for even $k\ge6$, the two constructions correspond under
\[
\Theta\wedge ds
\longleftrightarrow
\iota_{\rho\pr_\rho}\kappa
\wedge\iota_{\rho\pr_\rho}\omega,
\qquad
\Theta\wedge dz
\longleftrightarrow
\iota_{\rho\pr_\rho}\kappa\wedge dz.
\]
\end{remark}

\begin{proposition}\label{prop:h4-criterion}
The functions \eqref{eq:h4-A-B} satisfy
\begin{equation}\label{eq:h4-coefficient-equations}
    \rho A_\rho+6A=\frac1h, \quad \rho B_\rho+4B=\frac1h, 
\end{equation}
and
\begin{equation}\label{eq:h4-coefficient-inequalities}
    16B^2+24\rho^2A^2\leq1.
\end{equation}
With \eqref{eq:h4-coefficient-equations} and \eqref{eq:h4-coefficient-inequalities}, $\xi$ is a smooth calibration form on $\mathbb R^9$, and calibrates $H_4$.  Consequently, $H_4$ is area-minimizing.
\end{proposition}

\begin{proof}
The verifications of \eqref{eq:h4-coefficient-equations} and \eqref{eq:h4-coefficient-inequalities} are direct computations.

\textit{Step 1: smoothness.} The function $h=(1+|x|^2+|y|^2)^{1/2}$ is smooth and positive on $\R^8$, so $A$ and $B$ in \eqref{eq:h4-A-B} are smooth; and $\rho\pr_\rho=(x,y)$ is the Euler field of $\R^8$, so the two contractions in \eqref{eq:h4-potential} have polynomial coefficients. Hence $\Psi$, and therefore $\xi=d\Psi$, is smooth on $\R^9$.

\textit{Step 2: restriction to $H_4$.}
A straightforward computation using Cartan's formula shows that
\[
 d\iota_{\rho\pr_\rho}\kappa=4\kappa,
 \qquad d\iota_{\rho\pr_\rho}\omega=2\omega.
\]
Taking the exterior derivative on \eqref{eq:h4-potential} gives that
\begin{align}\label{eq:h4-xi-expansion}\begin{split}
    \xi &= A_\rho\,d\rho\wedge\iota_{\rho\pr_\rho}\kappa\wedge\iota_{\rho\pr_\rho}\omega + 4A\,\kappa\wedge\iota_{\rho\pr_\rho}\omega - 2A\,\iota_{\rho\pr_\rho}\kappa\wedge\omega \\
    &\quad + B_\rho\,d\rho\wedge\iota_{\rho\pr_\rho}\kappa\wedge dz + 4B\,\kappa\wedge dz.
\end{split}\end{align}

Contracting \eqref{eq:h4-xi-expansion} against the frame \eqref{eq:h4-tangent-splitting} is a straightforward computation:
\[
 \xi(P_z,(-y,x)+\pr_z)=\rho^2(\rho A_\rho+6A)+(\rho B_\rho+4B).
\]
A useful observation for the computation is that at any point of $H_4$, $\rho\pr_\rho$ belongs to the ruling $P_z$. Due to \eqref{eq:h4-coefficient-equations}, the preceding expression equals $h$. Together with \eqref{eq:h4-tangent-volume}, this gives $\left.\xi\right|_{H_4}=d\mu_{H_4}$. We retain the two identities in \eqref{eq:h4-coefficient-equations} separately because both are used in the comass estimate.

\textit{Step 3: the comass estimate.} Put $e=(x,y)/\rho$ and write $\kappa$ by using the decomposition $\Lambda^4\R^8=\Lambda^4e^\perp\oplus e^\flat\wedge\Lambda^3e^\perp$:
\[
 \kappa=e^\flat\wedge\alpha+\beta,\qquad \alpha=\iota_e\kappa,\qquad \iota_e\beta=0 .
\]
Since the form $\kappa$ in \eqref{eq:Kraines-4-form} is self-dual,
\begin{align} \label{eqn-Kraines-sd-1}
    *_8(e^\flat\wedge\alpha)=\beta,\qquad *_8\beta=e^\flat\wedge\alpha .
\end{align}
Combining these with the identity $*_8(\,\cdot\,\wedge v^\flat)=\iota_v\left(*_8\,\cdot\,\right)$ gives, for every unit vector $v$ orthogonal to $e$,
\begin{align} \label{eqn-Kraines-sd-2}
    *_8(e^\flat\wedge\alpha\wedge v^\flat)=\iota_v\beta,\qquad
    *_8(\beta\wedge v^\flat)=\iota_v\left(e^\flat\wedge\alpha\right)=-e^\flat\wedge\iota_v\alpha.
\end{align}
Let $\mathcal J(x,y)=(-y,x)$ be the orthogonal complex structure whose K\"ahler form is $\omega$; $\mathcal J$ commutes with the quaternionic multiplication on $\mathbb R^8 = \mathbb H\oplus\mathbb H$. Consider
\[
 \mathcal Je=\frac{(-y,x)}\rho.
\]
Note that $\mathcal Je$ is a unit vector orthogonal to $e$. Decomposing $\mathcal Je$ with respect to $(\mathbb He)^\perp\oplus\mathbb He$ therefore gives, for a suitable $I$ in the $S^2$-family of quaternionic complex structures,  a unit vector $q\in (\mathbb He)^\perp$ and an angle $0\leq t\leq\frac{\pi}{4}$ such that
\begin{equation}\label{eq:h4-screw-decomposition}
 \mathcal Je = \cos(2t)\,q-\sin(2t)\,Ie,\qquad \sin(2t)=\frac{2\left|\operatorname{Im}(x\bar y)\right|}{\rho^2}.
\end{equation}
Since $|\operatorname{Im}(x\bar y)|^2 =|x|^2|y|^2-\langle x,y\rangle^2$, this is the same angular variable $t$ as in \eqref{eq:rho-t-in-xy}. The set where $\rho>0$ and $t<\frac\pi4$ is the complement in $\R^9$ of the zero set of the $z$-independent nonzero polynomial $(|x|^2+|y|^2)^2-4|\operatorname{Im}(x\bar y)|^2$, hence open and dense.  Since $\|\xi\|_*$ is a continuous function of the point, it suffices to prove $\|\xi\|_*\leq1$ on this set.  Fix such a point and write $L=\mathbb He$, so that $L^\perp=\mathbb Hq$; let $\omega_I$ be the K\"ahler form of $I$ and set
\[
 \varsigma =\frac23\left.\omega_I\right|_{\langle e,Ie,q,Iq\rangle^\perp}.
\]

\textit{(a) An adapted frame.} Complete $I$ to an oriented orthonormal triple $I_1=I$, $I_2$, $I_3$ in the $S^2$-family, so that $I_1I_2=I_3$, and for the rest of Step 3 let $\omega_1,\omega_2,\omega_3$ denote the K\"ahler forms of this triple; the expression $\kappa=\frac16(\omega_1^2+\omega_2^2+\omega_3^2)$ is invariant under such rotations of the triple. Since $q\perp L$ and the $I_i$ are isometries preserving $L$ and $L^\perp$,
\[
    e_0=e,\quad e_a=I_ae,\quad f_0=q,\quad f_a=I_aq\qquad(a=1,2,3)
\]
is an orthonormal frame of $\R^8$ in which each $\omega_i$ takes its standard coordinate form. Each $I_i$ preserves $L$, so $\omega_i=\omega_i^L+\omega_i^{L^\perp}$, and squaring gives
\begin{equation}\label{eq:h4-kappa-blocks}
    \kappa=\mathrm{vol}_L+\mathrm{vol}_{L^\perp}+\frac13\sum_{i=1}^3\omega_i^L\wedge\omega_i^{L^\perp}.
\end{equation}
Write $\omega_i^L=e^{0i}+\varepsilon_i$ and $\omega_i^{L^\perp}=f^{0i}+\tau_i$, where $(\varepsilon_1,\varepsilon_2,\varepsilon_3)=(e^{23},e^{31},e^{12})$ and $\tau_i$ is the same expression in the $f$'s. Then \eqref{eq:h4-kappa-blocks} gives
\[
    \alpha=e^{123}+\frac13\sum_ie^i\wedge\omega_i^{L^\perp},\qquad
    \beta=\mathrm{vol}_{L^\perp}+\frac13\sum_i\varepsilon_i\wedge\omega_i^{L^\perp},
\]
and
\[
    \iota_q\beta=f^{123}+\frac13\sum_i\varepsilon_i\wedge f^i,\qquad
    \beta_q:=\beta-q^\flat\wedge\iota_q\beta=\frac13\sum_i\varepsilon_i\wedge\tau_i .
\]
It follows that $\omega_1=e^{01}+\varepsilon_1+f^{01}+\tau_1$, $\langle e,Ie,q,Iq\rangle^\perp=\langle e_2,e_3,f_2,f_3\rangle$, and
    \[
        \varsigma=\frac23\left(\varepsilon_1+\tau_1\right),\qquad\text{whence}\quad\|\varsigma\|_*=\frac23 .
    \]

\textit{(b) A useful identity.} The following identity will be useful for the estimate of the comass:
\begin{equation}\label{eq:theta-identity}
    e^\flat\wedge\left(\omega+3\,\iota_{\mathcal Je}\alpha\right)=-3\sin(2t)\,e^\flat\wedge\varsigma,
\end{equation}
where $e^\flat=e^0$ in the frame introduced in (a). To prove this identity, first observe that \eqref{eq:h4-screw-decomposition} implies
\[
    \mathcal Je_1=\cos(2t)\,f_1+\sin(2t)\,e,\quad
    \mathcal Je_2=\cos(2t)\,f_2+\sin(2t)\,e_3,\quad
    \mathcal Je_3=\cos(2t)\,f_3-\sin(2t)\,e_2 .
\]
Since $t<\frac\pi4$, it is not hard to find that $\mathcal Jq=-\cos(2t)\,e+\sin(2t)\,f_1$, and hence,
\[
    \mathcal Jf_1=-\cos(2t)\,e_1-\sin(2t)\,q,\quad
    \mathcal Jf_2=-\cos(2t)\,e_2-\sin(2t)\,f_3,\quad
    \mathcal Jf_3=-\cos(2t)\,e_3+\sin(2t)\,f_2 .
\]
These relations allow us to write down $\omega$ in terms of the basis $\{e_0,\ldots,e_3,f_0,\cdots,f_3\}$:
\[
    \omega=\cos(2t)\sum_{a=0}^3e^a\wedge f^a+\sin(2t)\left(-e^0\wedge e^1+f^0\wedge f^1+e^2\wedge e^3-f^2\wedge f^3\right).
\]

On the other hand, contracting the formulas of (a) gives $\iota_q\alpha=-\frac13\sum_{i=1}^3e^i\wedge f^i$ and $\iota_{Ie}\alpha=e^{23}+\frac13\left(f^{01}+f^{23}\right)$, so by \eqref{eq:h4-screw-decomposition},
\[
    \omega+3\,\iota_{\mathcal Je}\alpha=\cos(2t)\,e^0\wedge f^0-\sin(2t)\,e^0\wedge e^1-2\sin(2t)\left(e^{23}+f^{23}\right);
\]
wedging the formula with $e^\flat = e^0$ proves \eqref{eq:theta-identity}.

\textit{(c) The Hodge star of $\xi$.} Since $d\rho=e^\flat$, $\iota_{\rho\pr_\rho}\kappa=\rho\alpha$, $\iota_{\rho\pr_\rho}\omega=\rho (\mathcal Je)^\flat$ and $\kappa\wedge\omega=-*_8\omega$, \eqref{eq:h4-xi-expansion} can be rewritten as
\[
    \xi=\rho^2A_\rho\,e^\flat\wedge\alpha\wedge (\mathcal Je)^\flat+6\rho A\,\kappa\wedge (\mathcal Je)^\flat+2\rho A\,\iota_e\left(*_8\omega\right)
    +\rho B_\rho\,e^\flat\wedge\alpha\wedge dz+4B\,\kappa\wedge dz.
\]
Equip $\R^8\oplus\R\pr_z$ with the product orientation.  By \eqref{eqn-Kraines-sd-1}, \eqref{eqn-Kraines-sd-2}, $\iota_e*_8\omega=*_8(\omega\wedge e^\flat)$, and writing $\beta=(\mathcal Je)^\flat\wedge\iota_{\mathcal Je}\beta+\beta_{\mathcal Je}$ with $\iota_{\mathcal Je}\beta_{\mathcal Je}=0$, the Hodge star of $\xi$ is
\begin{align}
    *_9\xi &= (\rho^2 A_\rho+6\rho A)\,\iota_{\mathcal Je}\beta\wedge dz - 6\rho A\,e^\flat\wedge\iota_{\mathcal Je}\alpha\wedge dz - 2\rho A\,\omega\wedge e^\flat\wedge dz \notag \\
    &\quad + \rho B_\rho\,\beta + 4B\,\kappa \notag \\
    &= 4B\,e^\flat\wedge\alpha + (4B + \rho B_\rho)\,\beta - \rho\left(\rho A_\rho+6A\right)dz \wedge\iota_{\mathcal Je}\beta + 2\rho A\,dz\wedge e^\flat\wedge\left(\omega+3\,\iota_{\mathcal Je}\alpha\right). \label{eq:h4-hodge-general}
\end{align}
Applying \eqref{eq:h4-coefficient-equations} and \eqref{eq:theta-identity} to \eqref{eq:h4-hodge-general}, we obtain
\begin{equation}\label{eq:h4-hodge-formula}
 *_9\xi=4B\,e^\flat\wedge\alpha+\frac1h\beta-dz\wedge\left(\frac\rho h\iota_{\mathcal Je}\beta+6\rho A\sin(2t)\,e^\flat\wedge\varsigma\right).
\end{equation}

\textit{(d) The reduction.} 
Consider the horizontal--vertical decomposition, $\R^8\oplus\R\pr_z$.  For any oriented unit simple \emph{four}-plane $\Pi$ in $\R^8\oplus\R\pr_z$, $\Pi\cap(\R^8\oplus\{0\})$ has dimension at least three.  Choose an oriented orthonormal three-vector $\Pi_h$ in this intersection.  The remaining unit vector of $\Pi$ can then be written $\cos\theta\,v+\sin\theta\,\pr_z$, where $v$ is horizontal, unit, and orthogonal to $\Pi_h$.

For any fixed $\Pi_h$, it follows from \eqref{eq:h4-hodge-formula} that
\begin{align*}
    &\quad \max_{v,\theta}\bigl[(*_9\xi)\,(\Pi_h\wedge(\cos\theta\,v+\sin\theta\,\pr_z))\bigr]^2 \\
    &= \left|\iota_{\Pi_h}\left(4B\,e^\flat\wedge\alpha+\frac1h\beta\right)\right|^2 + \left[\frac\rho h\left(\iota_{\mathcal Je}\beta\right)(\Pi_h)+6\rho A\sin(2t)(e^\flat\wedge\varsigma)(\Pi_h)\right]^2 .
\end{align*}
By expanding $\mathcal Je $ by \eqref{eq:h4-screw-decomposition} and using the basic inequality $(c_0\cos(2t)+c_1\sin(2t))^2\leq c_0^2+c_1^2$, it suffices to show that
\begin{equation}\label{eq:h4-comass-reduction}
 \left|\iota_{\Pi_h}\left(4B\,e^\flat\wedge\alpha+\frac1h\beta\right)\right|^2+\frac{\rho^2}{h^2}(\iota_q\beta)(\Pi_h)^2+\left[-\frac\rho h(\iota_{Ie}\beta)(\Pi_h)+6\rho A(e^\flat\wedge\varsigma)(\Pi_h)\right]^2\leq1
\end{equation}
for every unit simple horizontal three-vector $\Pi_h$.

To proceed, note that $\Pi_h\cap e^\perp$ has dimension at least two.  Thus, we can choose\footnote{The notation $v$ is used in three different places of this proof; the $v$'s in different parts are unrelated.} $u,v,w\in e^\perp$ and $\Pi_h = (\cos\varphi\,e+\sin\varphi\,u)\wedge v\wedge w$ for some angle $\varphi$.  With this understood, we fix any orthonormal triple $u,v,w\in e^\perp$, and analyze the maximum of the left-hand side of \eqref{eq:h4-comass-reduction} over $\varphi$, or more precisely, over $\Pi_h(\varphi) = (\cos\varphi\,e+\sin\varphi\,u)\wedge v\wedge w$.

\textit{(e) Some inequalities.} 
Let
\[ \mathcal A=\alpha(v,w,\cdot) \quad\text{and}\quad \mathcal B=\beta(u,v,w,\cdot). \]
Since $\iota_e\alpha=0$ and $\iota_e\beta=0$, both annihilate $e$.

Since $\kappa$ has comass 1, $\bigl|\iota_{\Pi_h(\varphi)}\kappa\bigr|^2\leq 1$ for any $\varphi$.  With\footnote{The convention is $\iota_{X_1\wedge X_2\wedge X_3}\Phi=\Phi(X_1,X_2,X_3,\cdot)$.}
\[ \iota_{\Pi_h(\varphi)}\kappa = \cos\varphi\,\mathcal A + \sin\varphi\left(-\alpha(u,v,w)\,e^\flat+\mathcal B\right), \]
we find that
\begin{align} \label{mathcal-A-B-bound} \begin{split}
    |\mathcal A|\leq1,\qquad \alpha(u,v,w)^2+|\mathcal B|^2\leq1,\\
    \langle\mathcal A,\mathcal B\rangle^2\leq\left(1-|\mathcal A|^2\right)\left(1-\alpha(u,v,w)^2-|\mathcal B|^2\right).
\end{split} \end{align}

We need two further inequalities:
\begin{equation}\label{eq:h4-three-vector-bounds}
 |\varsigma(v,w)|\leq\frac23,\qquad \mathcal B(q)^2+\alpha(u,v,w)^2+3\mathcal B(Ie)^2\leq1.
\end{equation}
The first one follows from the definition of $\varsigma$, and the second one is precisely \eqref{eq:h4-fixed-estimate}, proved in Lemma~\ref{lem:h4-fixed} (applied to $u\wedge v\wedge w$).

\textit{(f) The estimate.} 
Direct contraction gives
\begin{align*}
    \iota_{\Pi_h(\varphi)}\left(4B\,e^\flat\wedge\alpha+\tfrac1h\beta\right)&=4B\cos\varphi\,\mathcal A+\sin\varphi\left(-4B\,\alpha(u,v,w)\,e^\flat+\tfrac1h\mathcal B\right),\\
    \tfrac\rho h\,(\iota_q\beta)(\Pi_h(\varphi))&=-\tfrac{\rho\sin\varphi}h\,\mathcal B(q),\\
    -\tfrac\rho h\,(\iota_{Ie}\beta)(\Pi_h(\varphi))+6\rho A\,(e^\flat\wedge\varsigma)(\Pi_h(\varphi))&=\tfrac{\rho\sin\varphi}h\,\mathcal B(Ie)+6\rho A\cos\varphi\,\varsigma(v,w).
\end{align*}

Consider the following map
\[      
    (\cos\varphi,\sin\varphi) \mapsto \left(\iota_{\Pi_h(\varphi)}\left(4B\,e^\flat\wedge\alpha+\tfrac1h\beta\right),\
    \tfrac\rho h(\iota_q\beta)(\Pi_h(\varphi)),\
    -\tfrac\rho h(\iota_{Ie}\beta)(\Pi_h(\varphi))+6\rho A\,(e^\flat\wedge\varsigma)(\Pi_h(\varphi))\right),
\]
which takes values in $\Lambda^1\R^8\oplus\R^2$.  Clearly, \eqref{eq:h4-comass-reduction} holds if and only if the squared norm of its value is at most $1$.  By regarding it as a quadratic form on $\R^2$, it has coefficients
\begin{align*}
    \mathsf{E} &= 16B^2|\mathcal A|^2+36\rho^2A^2\,\varsigma(v,w)^2,\\
    \mathsf{G} &= 16B^2\alpha(u,v,w)^2 + \frac1{h^2}|\mathcal B|^2 + \frac{\rho^2}{h^2}\left(\mathcal B(q)^2 + \mathcal B(Ie)^2\right),\\
    \mathsf{F} &= \frac{4B}h\langle\mathcal A,\mathcal B\rangle + \frac{6\rho^2A}h\,\mathcal B(Ie)\,\varsigma(v,w).
\end{align*}
Therefore, it remains to prove the following:
\begin{align*}
    \mathsf{E} \leq 1 , \quad \mathsf{G} \leq 1 , \quad\text{and }\; \mathsf{F}^2 \leq (1-\mathsf{E})(1-\mathsf{G}).
\end{align*}

For $\mathsf{E}\leq1$: according to \eqref{eq:h4-coefficient-inequalities}, \eqref{mathcal-A-B-bound} and \eqref{eq:h4-three-vector-bounds},
\begin{align*}
    1-\mathsf{E}
    &= 16B^2\left(1-|\mathcal A|^2\right) + \left(1-16B^2-24\rho^2A^2\right) + 36\rho^2A^2\left(\frac23-\varsigma(v,w)^2\right) \\
    &\geq 16B^2\left(1-|\mathcal A|^2\right) + 36\rho^2A^2\left(\frac23-\varsigma(v,w)^2\right) \geq 0.
\end{align*}
For $\mathsf{G}\leq1$: using $1=h^{-2}+\rho^2h^{-2}$, \eqref{eq:h4-coefficient-inequalities} (which gives $16B^2\leq1$) and \eqref{eq:h4-three-vector-bounds},
\begin{align*}
    1-\mathsf{G}
    &=\frac1{h^2}\left(1-\alpha(u,v,w)^2-|\mathcal B|^2\right)+\frac{\rho^2}{h^2}\left(1-\mathcal B(q)^2-\mathcal B(Ie)^2\right)\\
    &\qquad-\left(16B^2-\frac1{h^2}\right)\alpha(u,v,w)^2\\
    &\geq\frac1{h^2}\left(1-\alpha(u,v,w)^2-|\mathcal B|^2\right)+ \frac{\rho^2}{h^2}\left(1-\mathcal B(q)^2-\mathcal B(Ie)^2-\alpha(u,v,w)^2\right)\\
    &\geq\frac1{h^2}\left(1-\alpha(u,v,w)^2-|\mathcal B|^2\right)+\frac{2\rho^2}{h^2}\mathcal B(Ie)^2\geq0.
\end{align*}
Note that $|\varsigma(v,w)|\leq\frac23$ leads to $\varsigma(v,w)^2\leq2\left(\frac23-\varsigma(v,w)^2\right)$.  By the Cauchy--Schwarz inequality and \eqref{mathcal-A-B-bound}, 
\begin{align*}
    |\mathsf{F}|&\leq\frac{4|B|}{h}\sqrt{(1-|\mathcal A|^2)(1-\alpha(u,v,w)^2-|\mathcal B|^2)}+\frac{6\rho^2|A|}{h}\left|\mathcal B(Ie)\,\varsigma(v,w)\right|\\
    &\leq\sqrt{16B^2(1-|\mathcal A|^2)}\sqrt{\frac1{h^2}(1-\alpha(u,v,w)^2-|\mathcal B|^2)}\\
    &\qquad+\sqrt{36\rho^2A^2\left(\frac23-\varsigma(v,w)^2\right)}\sqrt{\frac{2\rho^2}{h^2}\mathcal B(Ie)^2}\\
    &\leq\left[16B^2(1-|\mathcal A|^2)+36\rho^2A^2\left(\frac23-\varsigma(v,w)^2\right)\right]^{1/2}\\
    &\qquad\times\left[\frac1{h^2}(1-\alpha(u,v,w)^2-|\mathcal B|^2)+\frac{2\rho^2}{h^2}\mathcal B(Ie)^2\right]^{1/2}\\
    &\leq\sqrt{1-\mathsf{E}}\,\sqrt{1-\mathsf{G}}.
\end{align*}
Hence, $\xi=d\Psi$ is smooth and closed, $\|\xi\|_*\leq1$, and $\left.\xi\right|_{H_4}=d\mu_{H_4}$, so $\xi$ calibrates $H_4$. Consequently, $H_4$ is area-minimizing.
\end{proof}

The lemma below establishes an inequality claimed in Step 3 (e).

\begin{lemma}\label{lem:h4-fixed}
Let $e\in\R^8$ be a unit vector and write the $4$-form $\kappa$ in \eqref{eq:Kraines-4-form} as
\[
 \kappa=e^\flat\wedge\alpha+\beta,\qquad \alpha=\iota_e\kappa,\qquad \iota_e\beta=0.
\]
Let $I$ be any quaternionic complex structure and let $q$ be any unit vector orthogonal to the quaternionic line $\mathbb He$. Then, for every unit simple three-vector $\eta\in\Lambda^3e^\perp$,
\begin{equation}\label{eq:h4-fixed-estimate}
 (\iota_q\beta)(\eta)^2+\alpha(\eta)^2+3(\iota_{Ie}\beta)(\eta)^2\leq1.
\end{equation}
\end{lemma}

\begin{proof}
Complete $I$ to a hyperk\"ahler triple $I,J,K$, with K\"ahler forms $\omega_I,\omega_J,\omega_K$, and set
\[
 \Phi_1=\omega_I^2-3\kappa,\qquad \Phi_2=\omega_J^2-3\kappa,\qquad \Phi_3=\omega_K^2-3\kappa .
\]
These are Cayley calibrations, and $\omega_I^2+\omega_J^2+\omega_K^2=6\kappa$ gives $\Phi_1+\Phi_2+\Phi_3=-3\kappa$. Writing $\gamma_i=\iota_\eta\Phi_i$ and contracting this identity with $\eta$,
\[
 \frac13\sum_{i=1}^3\gamma_i(X)=-\kappa(\eta,X)=\kappa(X,\eta)\qquad\text{for every }X\in\R^8 .
\]
Since $\eta\in\Lambda^3e^\perp$, the decomposition $\kappa=e^\flat\wedge\alpha+\beta$ gives $\kappa(e,\eta)=\alpha(\eta)$, $\kappa(q,\eta)=(\iota_q\beta)(\eta)$ and $\kappa(Ie,\eta)=(\iota_{Ie}\beta)(\eta)$; taking $X=e$ and $X=q$ above therefore yields
\[
 \frac13\sum_{i=1}^3\gamma_i(e)=\alpha(\eta),\qquad \frac13\sum_{i=1}^3\gamma_i(q)=(\iota_q\beta)(\eta).
\]
For the third quantity only $\Phi_1$ is needed: since $\iota_{Ie}\omega_I=-e^\flat$, one has $\iota_{Ie}\omega_I^2=2\left(\iota_{Ie}\omega_I\right)\wedge\omega_I=-2\,e^\flat\wedge\omega_I$, which annihilates $\Lambda^3e^\perp$, so $\omega_I^2(\eta,Ie)=0$ and hence
\[
 \gamma_1(Ie)=\Phi_1(\eta,Ie)=-3\kappa(\eta,Ie)=3(\iota_{Ie}\beta)(\eta).
\]

Each $\Phi_i$ has comass one and $\eta$ is a unit simple three-vector, so $|\gamma_i|\leq1$. By Cauchy--Schwarz, $\alpha(\eta)^2\leq\frac13\sum_i\gamma_i(e)^2$ and $(\iota_q\beta)(\eta)^2\leq\frac13\sum_i\gamma_i(q)^2$, while $e$, $q$ and $Ie$ are orthonormal, so $\gamma_i(e)^2+\gamma_i(q)^2+\gamma_i(Ie)^2\leq|\gamma_i|^2$ for each $i$. Therefore,
\[
 (\iota_q\beta)(\eta)^2+\alpha(\eta)^2+3(\iota_{Ie}\beta)(\eta)^2
 \leq\frac13\sum_{i=1}^3\left(\gamma_i(q)^2+\gamma_i(e)^2\right)+\frac13\gamma_1(Ie)^2
 \leq\frac13\sum_{i=1}^3|\gamma_i|^2\leq1 .
\]
This finishes the proof of this lemma.
\end{proof}

\section{Construction of competitors for \texorpdfstring{$H_k$}{Hk} with odd \texorpdfstring{$k$}{k}}\label{sec:competitor-Hk-odd}

We write $(x,y,z)$ for the coordinates on $\R^k\times\R^k\times\R$.  By \eqref{eq:helicoid-parametrization}, $H_k$ is parametrized by $ F_k(u, s)=(\cos s\, u, \sin s \, u, s)$, where $u\in \R^k$ and $s\in \R$. 

\begin{theorem}\label{thm:odd-nonminimizing}
If $k$ is odd, then $H_k$ is not area-minimizing.
\end{theorem}

\subsection{The oriented replacement}\label{sec:odd-replacement}

Orient $H_k$ by the coordinates $(u_1,\ldots,u_k,s)$. In polar coordinates $u=\rho w$, the induced metric and volume form are
\begin{equation}\label{eq:odd-helicoid-geometry}
 g_{H_k}=d\rho^2+(1+\rho^2)ds^2+\rho^2|dw|^2,
 \qquad
 d\mu_{H_k}=\rho^{k-1}\sqrt{1+\rho^2}\,d\rho\,d\mu_{S^{k-1}}\,ds.
\end{equation}

Fix $R>0$ and write $B_R=\{u\in\R^k:|u|\leq R\}$. Consider the following portion of $H_k$:
\begin{equation}\label{eq:PR}
P_R=\{F_k(u,s):u\in B_R,\ 0\leq s\leq2\pi\}.    
\end{equation}
The last coordinate makes $F_k$ injective on $B_R\times[0,2\pi]$. By \eqref{eq:odd-helicoid-geometry},
\begin{equation}\label{eq:removed-volume}
 \Vol(P_R)=2\pi\Vol(S^{k-1})\int_0^R\rho^{k-1}\sqrt{1+\rho^2}\,d\rho.
\end{equation}
Here $\Vol$ denotes Hausdorff measure in the dimension of the set under consideration, and $\Vol(S^0)=2$.

Let $Q_R=B_R\times[0,\pi]^2\subset\R^{k+2}$, with coordinates $(u,s,t)$, and define
\[
 G:Q_R\longrightarrow\R^{2k+1},\qquad
 G(u,s,t)=(\cos s\,u,\sin s\,u,s+t).
\]
The $(k+2)$-dimensional Jacobian $J_{k+2}G$ measures the local change of $(k+2)$-dimensional volume from $Q_R$ to its image $G(Q_R)\subset\R^{2k+1}$.
Away from $u=0$, the map $G$ is injective except for the codimension-two identification $(u,0,\pi)\sim(-u,\pi,0)$. At $u=0$, its rank drops to $k+1$. Indeed,
\[
 G_{u_j}=(\cos s\,e_j,\sin s\,e_j,0),\quad
 G_s=(-\sin s\,u,\cos s\,u,1),\quad G_t=(0,0,1).
\]
The vectors $G_{u_1},\ldots,G_{u_k},G_s-G_t,G_t$ are mutually orthogonal, with lengths $1,\ldots,1,|u|,1$. Hence
\[
 J_{k+2}G
 =|G_{u_1}\wedge\cdots\wedge G_{u_k}\wedge(G_s-G_t)\wedge G_t|
 =|u|.
\]
The set $u=0$ maps under $G$ to the vertical axis and contributes no $(k+2)$-dimensional mass. The images of the faces $t=0,\pi$ under $G$ together form $P_R$:
\[
 G(u,s,0)=F_k(u,s),\qquad G(u,s,\pi)=F_k(-u,s+\pi).
\]
Under $G$, the faces $s=0$ and $s=\pi$ together map onto the cylinder
\[
 C_R=\{(u,0,z):u\in B_R,\ 0\leq z\leq2\pi\},
\]
and the face $|u|=R$ maps onto
\[
 L_R=\{(R\cos s\,w,R\sin s\,w,z):0\leq s\leq\pi,\ s\leq z\leq s+\pi,
 \ w\in S^{k-1}\}.
\]

To determine the orientations of the two helicoid faces, write $(v,\sigma)$ for the coordinates on the target helicoid,
\[
 F_k(v,\sigma)=(\cos\sigma\,v,\sin\sigma\,v,\sigma).
\]
Orient $Q_R$ by
\[
 \Omega=dt\wedge du_1\wedge\cdots\wedge du_k\wedge ds,
 \quad\text{and denote}\quad
 \omega=du_1\wedge\cdots\wedge du_k\wedge ds.
\]
Using the outward-normal-first convention on $Q_R$'s boundary, the boundary orientations are $-\omega$ on $t=0$ and $+\omega$ on $t=\pi$. Thus the $t=0$ face carries the negative of the $H_k$ orientation. On $t=\pi$, let $\Phi(u,s)=(-u,s+\pi)$. Then $G(u,s,\pi)=F_k(\Phi(u,s))$ and
\[
 \Phi^*(dv_1\wedge\cdots\wedge dv_k\wedge d\sigma)=(-1)^k\omega.
\]
Since the boundary form there is $+\omega$, the $t=\pi$ face carries $(-1)^k$ times the $H_k$ orientation. Hence, when $k$ is odd, both helicoid faces carry the negative of the $H_k$ orientation.

In the following, we write $[\![A]\!]$ for the current of integration over an oriented set $A$, and use the standard notations $\partial$ and $\Mass$ for boundary and mass in geometric measure theory; see \cite[Chapter 6]{SimonGMT}.  Define $T_R=G_{\#}([\![Q_R]\!])$, where $G_{\#}$ denotes the pushforward of currents under $G$. Since $G$ is smooth on $Q_R$, $\partial T_R=G_{\#}(\partial[\![Q_R]\!])$, with no boundary term from $u=0$. For odd $k$, the preceding computation gives
\begin{equation}\label{eq:boundary}
 \partial T_R=-[\![P_R]\!]+[\![K_R]\!],
 \qquad
 K_R=C_R\cup L_R,
\end{equation}
where the pieces of $K_R$ carry the orientations induced from $Q_R$. Thus
\begin{equation}\label{eq:odd-competitor}
 [\![H_k]\!]+\partial T_R=[\![H_k]\!]-[\![P_R]\!]+[\![K_R]\!]
\end{equation}
is an integral-current competitor agreeing with $[\![H_k]\!]$ outside a compact set.

\subsection{The area comparison}\label{sec:odd-area-comparison}

The cylinder over $B_R$ has volume
\begin{equation}\label{eq:cylinder-volume}
 \Vol(C_R)=2\pi\Vol(S^{k-1})\int_0^R\rho^{k-1}\,d\rho
 =\frac{2\pi}{k}\Vol(S^{k-1})R^k.
\end{equation}
On $L_R$, use $(s,w,z)$ as coordinates. A unit tangent vector to $S^{k-1}$ is scaled by $R$, while the coordinate vectors in the $s$ and $z$ directions have lengths $R$ and $1$; these directions are mutually orthogonal. Therefore
\begin{equation}\label{eq:lateral-volume}
 \Vol(L_R)=\Vol(S^{k-1})\int_0^\pi\int_s^{s+\pi}R^k\,dz\,ds
 =\pi^2\Vol(S^{k-1})R^k.
\end{equation}
The inserted pieces intersect one another only in sets of zero $(k+1)$-dimensional measure: $C_R\cap L_R$ is $k$-dimensional. Hence, combining \eqref{eq:cylinder-volume} and \eqref{eq:lateral-volume}, we obtain
\begin{align}
 \Vol(K_R)
 &=\Vol(S^{k-1})\left(\pi^2+\frac{2\pi}{k}\right)R^k.
 \label{eq:inserted-volume}
\end{align}

The inserted pieces also meet the retained part of $H_k$ only in sets of zero $(k+1)$-dimensional measure. Indeed, $C_R\subset\{y=0\}$, while $H_k\cap\{y=0\}$ is contained in the axis together with the slices $s\in\pi\mathbb{Z}$ and is at most $k$-dimensional. Also,
\[
 H_k\cap\{|(x,y)|=R\}=F_k(RS^{k-1}\times\R)
\]
is $k$-dimensional, so the same holds for $L_R\cap H_k$. Consequently, there is no overlap or cancellation in the mass comparison.

By \eqref{eq:removed-volume} and $\sqrt{1+\rho^2}\geq\rho$, for $R>0$ we have
\begin{equation}\label{eq:removed-lower}
 \Vol(P_R)\geq\frac{2\pi}{k+1}\Vol(S^{k-1})R^{k+1}.
\end{equation}
Comparing \eqref{eq:removed-lower} and \eqref{eq:inserted-volume}, if
\begin{equation}\label{eq:R-threshold}
 R>(k+1)\left(\frac{\pi}{2}+\frac1k\right),
\end{equation}
then
\begin{equation}\label{eq:strict-volume-comparison}
 \Vol(K_R)<\Vol(P_R).
\end{equation}

\begin{proof}[Proof of Theorem~\ref{thm:odd-nonminimizing}]
Choose $R$ as in \eqref{eq:R-threshold} and set $S_R=\partial T_R$. This is a compactly supported integral $(k+1)$-cycle. For every ball $B$ containing its support, \eqref{eq:boundary}, \eqref{eq:odd-competitor}, the preceding observations on the intersections, and \eqref{eq:strict-volume-comparison} give
\[
 \Mass(([\![H_k]\!]+S_R)\mathbin{\llcorner}B)
 -\Mass([\![H_k]\!]\mathbin{\llcorner}B)
 =\Vol(K_R)-\Vol(P_R)<0.
\]
Thus $H_k$ is not area-minimizing. No smoothing is required, since piecewise-smooth integral currents are admissible competitors.
\end{proof}

\begin{remark}
The construction is especially easy to visualize when $k=1$; see Figures~\ref{fig:H1-3dpicture} and \ref{fig:H1-competitor}. The removed portion lies in the solid cylinder
\[
\{(x,y,z):\sqrt{x^2+y^2}\leq R,\ 0\leq z\leq2\pi\}.
\]
The replacement $K_R=C_R\cup L_R$ consists of three connected pieces: one rectangle $C_R$ and two bands on the cylinder
\[
\{(x,y,z):\sqrt{x^2+y^2}=R,\ 0\leq z\leq2\pi\}.
\]
Writing $w\in S^0=\{+1,-1\}$ and $\theta$ for the polar angle in the $(x,y)$-plane, one has $\theta=s$ on the component $w=+1$ and $\theta=s+\pi$ on the component $w=-1$. Thus the two bands in $L_R$ are
\[
\{0\leq\theta\leq\pi,\ \theta\leq z\leq\theta+\pi\}
\cup
\{\pi\leq\theta\leq2\pi,\ \theta-\pi\leq z\leq\theta\}.
\]
As $R\to\infty$, $\Vol(P_R)\sim2\pi R^2$, whereas
\[
\Vol(K_R)=4\pi R+2\pi^2R=O(R).
\]
\end{remark}

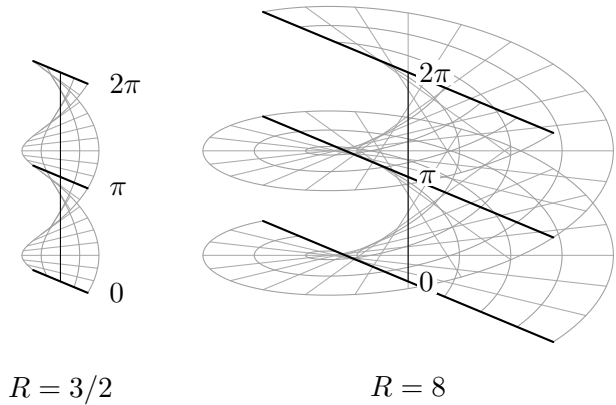
\begin{figure}[ht]
\centering
\begin{tikzpicture}[
    x={(0.24cm,-0.10cm)},y={(0.24cm,0.10cm)},z={(0cm,0.44cm)},
    line cap=round,line join=round,every node/.style={font=\small}
]
\begin{scope}
\foreach \a in {0,15,...,360}{
  \draw[gray!65,thin]
  ({-1.5*cos(\a)},{-1.5*sin(\a)},{pi*\a/180})--
  ({1.5*cos(\a)},{1.5*sin(\a)},{pi*\a/180});
}
\foreach \r in {-1.5,-0.75,0.75,1.5}{
  \draw[gray!75,thin] plot[domain=0:360,samples=121,variable=\a]
  ({\r*cos(\a)},{\r*sin(\a)},{pi*\a/180});
}
\draw (0,0,0)--(0,0,{2*pi});
\draw[thick] (-1.5,0,0)--(1.5,0,0);
\draw[thick] (-1.5,0,{pi})--(1.5,0,{pi});
\draw[thick] (-1.5,0,{2*pi})--(1.5,0,{2*pi});
\foreach \h/\label in {0/0,pi/\pi,2*pi/2\pi}{
  \node[right,inner sep=1pt,xshift=7pt] at (1.5,0,{\h}) {$\label$};
}
\node[below] at (0,0,-2.5) {$R=3/2$};
\end{scope}
\begin{scope}[shift={(4.6cm,0cm)}]
\foreach \a in {0,15,...,360}{
  \draw[gray!65,thin]
  ({-8*cos(\a)},{-8*sin(\a)},{pi*\a/180})--
  ({8*cos(\a)},{8*sin(\a)},{pi*\a/180});
}
\foreach \r in {-8,-6,-4,-2,2,4,6,8}{
  \draw[gray!75,thin] plot[domain=0:360,samples=121,variable=\a]
  ({\r*cos(\a)},{\r*sin(\a)},{pi*\a/180});
}
\draw (0,0,0)--(0,0,{2*pi});
\draw[thick] (-8,0,0)--(8,0,0);
\draw[thick] (-8,0,{pi})--(8,0,{pi});
\draw[thick] (-8,0,{2*pi})--(8,0,{2*pi});
\foreach \h/\label in {0/0,pi/\pi,2*pi/2\pi}{
  \node[right,fill=white,inner sep=1pt,xshift=3pt] at (0,0,{\h}) {$\label$};
}
\node[below] at (0,0,-2.5) {$R=8$};
\end{scope}
\end{tikzpicture}
\caption{Wireframe views of $P_R$ in \eqref{eq:PR} for $k=1$, with $R=3/2$ and $R=8$ at the same scale. The vertical axis is interior for $0<z<2\pi$.}
\label{fig:H1-3dpicture}
\end{figure}

\begin{figure}[ht]
\centering
\begin{tikzpicture}[
    scale=0.83,line cap=round,line join=round,
    every node/.style={font=\scriptsize},
    glue/.style={circle,draw,fill=white,inner sep=1pt,font=\scriptsize}
]
\begin{scope}
\fill[gray!22] (-2.3,0) rectangle (2.3,4);
\draw[->] (-2.65,0)--(2.9,0) node[right] {$x$};
\draw[->] (0,-0.25)--(0,4.5) node[above] {$z$};
\draw[thick] (-2.3,0)--(-2.3,4)--(2.3,4)--(2.3,0);
\draw[densely dashed,gray] (-2.3,2)--(2.3,2);
\node[below] at (-2.3,0) {$-R$};
\node[below] at (0,0) {$0$};
\node[below] at (2.3,0) {$R$};
\node[left,fill=gray!22,inner sep=1pt] at (0,4) {$2\pi$};
\node[left,fill=gray!22,inner sep=1pt] at (0,2) {$\pi$};
\node[glue] at (2.3,1) {$1$};
\node[glue] at (2.3,3) {$2$};
\node[glue] at (-2.3,1) {$3$};
\node[glue] at (-2.3,3) {$4$};
\node[font=\small] at (0,-0.8) {(a) the rectangle $C_R$};
\end{scope}
\begin{scope}[shift={(5.1,0)}]
\fill[gray!22] (0,0)--(2,2)--(2,4)--(0,2)--cycle;
\fill[gray!22] (2,0)--(4,2)--(4,4)--(2,2)--cycle;
\draw[->] (-0.25,0)--(4.5,0) node[right] {$\theta$};
\draw[->] (0,-0.25)--(0,4.5) node[above] {$z$};
\draw[thick] (0,0)--(2,2)--(4,4);
\draw[thick] (0,2)--(2,4);
\draw[thick] (2,0)--(4,2);
\draw (2,0)--(2,4);
\draw[densely dashed,gray] (4,0)--(4,2);
\draw (4,2)--(4,4);
\node[below] at (0,0) {$0$};
\node[below] at (2,0) {$\pi$};
\node[below] at (4,0) {$2\pi$};
\node[left] at (0,2) {$\pi$};
\node[left] at (0,4) {$2\pi$};
\node[rotate=45,fill=white,inner sep=1pt] at (1,1) {$z=\theta$};
\node[rotate=45,fill=white,inner sep=1pt] at (1,3) {$z=\theta+\pi$};
\node[rotate=45,fill=white,inner sep=1pt] at (3,1) {$z=\theta-\pi$};
\node[glue] at (0,1) {$1$};
\node[glue] at (4,3) {$2$};
\node[glue] at (2,1) {$3$};
\node[glue] at (2,3) {$4$};
\node[font=\small] at (2,-0.8) {(b) the cylinder, $\rho=R$};
\end{scope}
\end{tikzpicture}
\caption{The replacement $K_R=C_R\cup L_R$ for $k=1$: (a) the rectangle; (b) the two bands on the cylinder, with $\theta=0$ and $\theta=2\pi$ identified. Matching circled numbers mark glued boundary segments. The dashed line in (a) is the interior seam $z=\pi$.}
\label{fig:H1-competitor}
\end{figure}
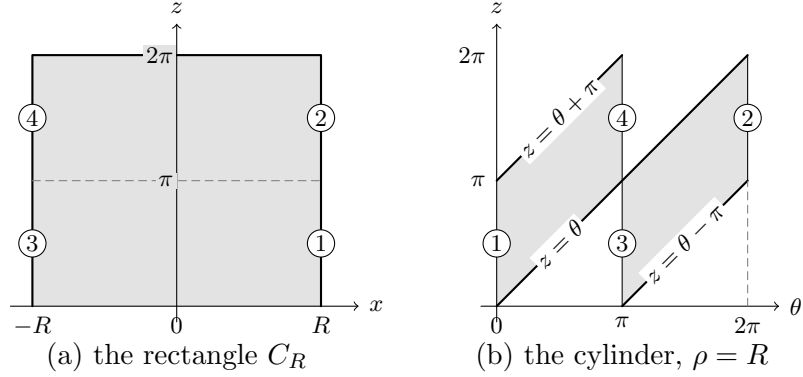

\begin{remark}
For even $k$, one has $(-1)^k=+1$, so the two helicoid faces occur in $\partial T_R$ with opposite $H_k$ orientations; the construction is therefore not an integral comparison. After reducing coefficients modulo $2$, the signs disappear, and the same construction shows that $H_k$ is not area-minimizing modulo $2$ for any $k$. The parity is also visible under the orthogonal projection $\R^{2k+1}\to\R^{2k}$ forgetting $z$, which identifies $(u,s)$ with $(-u,s+\pi)$ and maps the two consecutive turns to the same rank-one cone.
\end{remark}

\begin{remark}
In \cite{LawlorDirectedSlicing}, Lawlor's directed-slicing argument shows that the portion of the helicoid
\[
\{F_1(u,s):0\leq u\leq R,\ 0\leq s\leq6\pi\}
\]
is area-minimizing among oriented competitors with the same boundary. Our choice of the competitor $P_R$ in \eqref{eq:PR} is inspired by Lawlor's work. In Lawlor's example, the axis $u=0$ is part of the boundary. In contrast, it is interior to $P_R=F_1([-R,R]\times[0,2\pi])$ for $0<s<2\pi$. Thus these are two problems with different boundaries.
\end{remark}

\begin{remark}
    The competitor also works for certain generalized helicoids \eqref{eq:general-helicoid}:  Suppose that $\alpha>0$, $m_j\in\mathbb Z$, not all zero, and $\sum_{j=1}^k m_j$ is odd.  Let \(\bla = \alpha(m_1,\ldots,m_k)\).  Then, $\mathcal H_{\bla}$ is not area-minimizing.
\end{remark}

\section{Stability of generalized helicoids in slope space}\label{sec:generalized}

We return to the generalized helicoids $\mathcal{H}_{\bla}$ defined by the parametrization in \eqref{eq:general-helicoid}. Proposition~\ref{prop:stability-reduction} reduces their stability to nonnegativity of $q_{\bla}$ in \eqref{eq:L^2-functional}. For helicoids $\Hk$ with parameters $\bla=(1,\ldots,1)$, Proposition~\ref{prop:sharp-hardy} proves stability when $k\ge3$. We now allow the slopes to vary and derive stable and unstable regions in the parameter space $(0,\infty)^k$.

\subsection{Anisotropic stable regions}\label{sec:Hardy-anisotropic}

Let
\[
\bla=(\lambda_1,\ldots,\lambda_k),\qquad \Lambda=\operatorname{diag}(\lambda_1,\ldots,\lambda_k), \qquad \lambda_i>0.
\]
Under the change of variables
\[
 \tilde f(u)=f(\Lambda u),
 \qquad
 x=\Lambda u,
\]
$q_{\bla}$ in \eqref{eq:L^2-functional} becomes
\begin{equation}\label{eq:anisotropic-q}
 \Bigl(\prod_{i=1}^k\lambda_i\Bigr)q_{\bla}[\tilde f]
 =\int_{\R^k}h\sum_{j=1}^k\lambda_j^2\left[
 |\pr_{x_j}f|^2
 +\bigl|\pr_{x_j}(x\cdot f)\bigr|^2
 -\frac{3(f^j)^2}{h^2}
 \right]dx,
 \qquad
 h=\sqrt{1+|x|^2}.
\end{equation}
Hence $\mathcal{H}_{\bla}$ is stable if the right-hand side is nonnegative for every $f\in C_c^\infty(\R^k;\R^k)$.

For unequal slopes, Proposition~\ref{prop:sharp-hardy} cannot be applied directly to the integral in \eqref{eq:anisotropic-q}. Set
\[
 \lambda_{\min}=\min_j\lambda_j,
 \qquad
 \lambda_{\max}=\max_j\lambda_j.
\]
For $0\le\sigma\le\lambda_{\min}^2$, we split the terms involving derivatives into an isotropic part of size $\sigma$ and remaining terms. The isotropic part is controlled by Proposition~\ref{prop:sharp-hardy}, while the residual part is estimated through the scalar quantity introduced in the proof below. Optimizing over $\sigma$ gives the following criterion.

\begin{theorem}\label{thm:aniso-stability}
Let $k\ge3$ and $a_1,\ldots,a_k\geq0$, and set
\begin{equation}\label{eq:H-definition}
 \mu(a_1,\ldots,a_k)=\sup_{\tau\ge0}\left\{
 \frac12\sum_{j=1}^k\left(\sqrt{a_j^4+4a_j^2\tau}
 -a_j^2\right)-\tau\right\}.
\end{equation}
If $\bla=(\lambda_1,\ldots,\lambda_k)\in (0,\infty)^k$ satisfies
\begin{equation}\label{eq:aniso-criterion}
 \sup_{0\le\sigma\le\lambda_{\min}^2}\left\{
 \frac{(k-1)^2+4}{4}\,\sigma
 +\mu\Bigl(\sqrt{\lambda_1^2-\sigma},\ldots,
 \sqrt{\lambda_k^2-\sigma}\Bigr)\right\}
 \ \ge\ 2\lambda_{\max}^2 ,
\end{equation}
then $\mathcal{H}_{\bla}$ is stable.
\end{theorem}

\begin{proof}
Fix $\sigma\in[0,\lambda_{\min}^2]$, so that $\lambda_j^2-\sigma\ge0$
for every $j$. For any $f\in C_c^\infty(\R^k;\R^k)$, set $F=(f,x\cdot f)\in\R^{k+1}$,
$n=\frac{(-x,1)}{h}\in\R^{k+1}$ and $\rho=|F|$.  Since $F\cdot n=0$,
$|n|=1$ and $\pr_{x_j}F\cdot n=\frac{f^j}{h}$,
\begin{equation}\label{eq:rho-reduction}
 |\pr_{x_j}f|^2+|\pr_{x_j}(x\cdot f)|^2=|\pr_{x_j}F|^2
 \ \ge\ \left|\pr_{x_j}F\cdot\frac{F}{\rho}\right|^2
 +|\pr_{x_j}F\cdot n|^2
 =|\pr_{x_j}\rho|^2+\frac{(f^j)^2}{h^2}.
\end{equation}
Here $\rho$ is Lipschitz and \eqref{eq:rho-reduction} holds at every
point of $\{F\ne0\}$. On the other hand, on $\{F=0\}$ one has $f=0$ and $\nabla_x\rho=0$
almost everywhere, so \eqref{eq:rho-reduction} holds a.e.\ on $\R^k$,
and the integrations by parts below are performed on the compactly
supported Lipschitz function $\rho$.
Splitting the gradient sum at the level $\sigma$ and applying
\eqref{eq:rho-reduction} with the nonnegative weights
$\lambda_j^2-\sigma$,
\begin{align*}
 \sum_{j=1}^k\lambda_j^2\Bigl(|\pr_{x_j}f|^2
 +|\pr_{x_j}(x\cdot f)|^2\Bigr)
 \ \ge\ &\sigma\Bigl(|\nabla_xf|^2+|\nabla_x(x\cdot f)|^2\Bigr)\\
 &+\sum_{j=1}^k(\lambda_j^2-\sigma)|\pr_{x_j}\rho|^2
 +\sum_{j=1}^k(\lambda_j^2-\sigma)\frac{(f^j)^2}{h^2}.
\end{align*}
Multiplying by $h$, subtracting $\frac{3}{h}\sum_j\lambda_j^2(f^j)^2$, and using $\sum_j\bigl(2\lambda_j^2+\sigma\bigr)(f^j)^2
\le(2\lambda_{\max}^2+\sigma)|f|^2$, the integrand of
\eqref{eq:anisotropic-q} is bounded below by
\[
 h\,\sigma\Bigl(|\nabla_xf|^2+|\nabla_x(x\cdot f)|^2\Bigr)
 +h\sum_{j=1}^k(\lambda_j^2-\sigma)|\pr_{x_j}\rho|^2
 -\bigl(2\lambda_{\max}^2+\sigma\bigr)\frac{|f|^2}{h}.
\]
By Proposition~\ref{prop:sharp-hardy}, the first term integrates to at least
$c_k\sigma\int_{\R^k}\frac{|f|^2}{h}\,dx$.  For the second, let
$\tau\ge0$ and, for those $j$ with $\lambda_j^2>\sigma$, choose
$\gamma_j\ge0$ with
$(\lambda_j^2-\sigma)(\gamma_j^2+\gamma_j)=\tau$, that is
$\gamma_j=\frac{\sqrt{(\lambda_j^2-\sigma)^2
+4(\lambda_j^2-\sigma)\tau}-(\lambda_j^2-\sigma)}
{2(\lambda_j^2-\sigma)}$.  Using
$\pr_j\bigl(\frac{x_j}{h}\bigr)=\frac1h-\frac{x_j^2}{h^3}$ and
integrating by parts,
\begin{align*}
0&\le\int_{\R^k}h\sum_{j=1}^k(\lambda_j^2-\sigma)
\left|\pr_{x_j}\rho+\frac{\gamma_jx_j}{h^2}\rho\right|^2dx\\
&=\int_{\R^k}h\sum_{j=1}^k(\lambda_j^2-\sigma)|\pr_{x_j}\rho|^2\,dx
-\left(\sum_{j=1}^k(\lambda_j^2-\sigma)\gamma_j-\tau\right)
\int_{\R^k}\frac{\rho^2}{h}\,dx
-\tau\int_{\R^k}\frac{\rho^2}{h^3}\,dx ,
\end{align*}
so that, after taking the supremum over $\tau\ge0$ and recalling
\eqref{eq:H-definition},
\[
 \int_{\R^k}h\sum_{j=1}^k(\lambda_j^2-\sigma)|\pr_{x_j}\rho|^2\,dx
 \ \ge\ \mu\Bigl(\sqrt{\lambda_1^2-\sigma},\ldots,
 \sqrt{\lambda_k^2-\sigma}\Bigr)\int_{\R^k}\frac{\rho^2}{h}\,dx .
\]
Since $\mu\ge0$ and $\rho^2\ge|f|^2$, the right-hand side of
\eqref{eq:anisotropic-q} is therefore at least
\[
 \left[\frac{(k-1)^2+4}{4}\,\sigma
 +\mu\Bigl(\sqrt{\lambda_1^2-\sigma},\ldots,
 \sqrt{\lambda_k^2-\sigma}\Bigr)-2\lambda_{\max}^2\right]
 \int_{\R^k}\frac{|f|^2}{h}\,dx ,
\]
where we used $c_k\sigma-\sigma=\frac{(k-1)^2+4}{4}\sigma$.  Taking the
supremum over $\sigma$ and invoking \eqref{eq:aniso-criterion}
completes the proof.
\end{proof}

\begin{corollary}\label{cor:explicit-window}
Let $k\ge3$ and $\bla\in(0,\infty)^k$.  If
\begin{equation}\label{eq:stability-window}
 8\,\lambda_{\max}^2\ \le\ \bigl[(k-1)^2+4\bigr]\lambda_{\min}^2 ,
\end{equation}
then $\mathcal{H}_{\bla}$ is stable.
\end{corollary}

\begin{proof}
Take $\sigma=\lambda_{\min}^2$ in \eqref{eq:aniso-criterion} and use
$\mu\ge0$.
\end{proof}

Thus, for $k\ge4$, stability persists in an explicit neighborhood of the ray $\{\lambda_1=\cdots=\lambda_k>0\}$ in slope space.  At $k=3$, condition \eqref{eq:stability-window} forces $\lambda_{\max}=\lambda_{\min}$, while Remark~\ref{rem:explicit-unstable-regions} below shows that every non-isotropic parameter is unstable.  We next construct the unstable test fields and prove the openness of the unstable set.

\subsection{Unstable perturbations}\label{sec:unstable}

We first complete the instability half of Theorem~\ref{thm:stability} by treating $H_2$. Proposition~\ref{prop:stability-reduction} reduces the problem to finding a compactly supported field $ f$ with $q_{(1,1)}[f]<0$. The required field is tangent to the circles and supported on a large annulus.

For $x\in\R^2\setminus\{0\}$ write $x=r\omega$, $\omega=(\cos\theta,\sin\theta)$, and $e_\theta=(-\sin\theta,\cos\theta)$. The field on $S^1$ is
\[
Y(\theta)
=
\nabla_S\omega_1
=
e_1-\omega_1\omega
=
-\sin\theta\,e_\theta.
\]

\begin{proposition}\label{prop:degree-one-instability}
$\Hk$ is unstable for $k\le2$.  More precisely,
$\ind H_{2}=\infty$.
\end{proposition}

\begin{proof}
We first consider $k=2$.  Fix
$0\ne\phi\in C_c^\infty(\R)$ and, for $R>1$, define
\begin{equation}\label{eq:H2-annular-field}
f_R(x)
=
r^{-1/2}
\phi\left(\log\frac rR\right)
Y(\theta).
\end{equation}
Since $Y\cdot\omega=0$, we have $x\cdot f_R=0$.  Set
$\sigma=\log(r/R)$.  Using
\[
\pr_{x_1}r=\cos\theta,
\qquad
\pr_{x_2}r=\sin\theta,
\qquad
\pr_{x_1}\theta=-\frac{\sin\theta}{r},
\qquad
\pr_{x_2}\theta=\frac{\cos\theta}{r},
\]
together with
\[
\frac{dY}{d\theta}
=
-\cos\theta\,e_\theta+\sin\theta\,\omega,
\]
direct differentiation gives
\begin{align*}
\pr_{x_1}f_R
&=
r^{-3/2}
\left[
\left(
-\sin\theta\cos\theta\,\phi'
+\frac32\sin\theta\cos\theta\,\phi
\right)e_\theta
-\sin^2\theta\,\phi\,\omega
\right],\\
\pr_{x_2}f_R
&=
r^{-3/2}
\left[
\left(
-\sin^2\theta\,\phi'
+
\left(\frac12\sin^2\theta-\cos^2\theta\right)\phi
\right)e_\theta
+
\sin\theta\cos\theta\,\phi\,\omega
\right].
\end{align*}
Since $\{\omega,e_\theta\}$ is orthonormal,
\[
|\pr_{x_1}f_R|^2+|\pr_{x_2}f_R|^2
=
r^{-3}
\left[
\sin^2\theta\,(\phi')^2
-\sin^2\theta\,\phi\phi'
+
\left(
\frac54-\frac14\cos^2\theta
\right)\phi^2
\right].
\]
For $H_2$, formula \eqref{eq:anisotropic-q} therefore gives,
using $dx=r\,dr\,d\theta$ and $dr=r\,d\sigma$,
\begin{align*}
q_{(1,1)}[f_R]
&=
\int_\R\int_0^{2\pi}
\frac hr
\left[
\sin^2\theta\,(\phi')^2
-\sin^2\theta\,\phi\phi'
+
\left(
\frac54-\frac14\cos^2\theta
\right)\phi^2
\right]
d\theta\,d\sigma\\
&\quad
-
3\int_\R\int_0^{2\pi}
\frac rh
\sin^2\theta\,\phi^2
\,d\theta\,d\sigma,
\qquad
h=\sqrt{1+R^2e^{2\sigma}}.
\end{align*}
On the support of $\phi$, both $h/r$ and $r/h$ converge uniformly to
$1$ as $R\to\infty$.  Since
\[
\int_0^{2\pi}\sin^2\theta\,d\theta=\pi,
\qquad
\int_0^{2\pi}
\left(
\frac54-\frac14\cos^2\theta
\right)d\theta
=
\frac{9\pi}{4},
\]
and $\int_\R\phi\phi'\,d\sigma=0$, we obtain
\begin{equation}\label{eq:H2-annular-limit}
\lim_{R\to\infty}q_{(1,1)}[f_R]
=
\pi\int_\R(\phi')^2\,d\sigma
-
\frac{3\pi}{4}\int_\R\phi^2\,d\sigma .
\end{equation}
Now choose $0\ne\psi\in C_c^\infty(\R)$ and set
$\phi(\sigma)=\psi(\sigma/L)$.  Then
\[
\int_\R(\phi')^2\,d\sigma
=
\frac1L\int_\R(\psi')^2\,d\sigma,
\qquad
\int_\R\phi^2\,d\sigma
=
L\int_\R\psi^2\,d\sigma.
\]
Hence the right-hand side of \eqref{eq:H2-annular-limit} is negative
for all sufficiently large $L$.  Fixing such an $L$, we obtain
$q_{(1,1)}[f_R]<0$ for all sufficiently large $R$.  Therefore
Proposition~\ref{prop:stability-reduction} gives
$\ind H_{2}=\infty$.

For $k=1$, $H_1$ is the classical helicoid, a complete non-planar
minimal surface in $\R^3$, hence unstable by \cite{dCP,FCS}.
\end{proof}

Together, Theorem~\ref{thm:stable-half} and
Proposition~\ref{prop:degree-one-instability} prove
Theorem~\ref{thm:stability}.

The same annular perturbation also gives explicit unstable regions for anisotropic slopes.

\begin{remark}\label{rem:explicit-unstable-regions} The same annular computation applies in general dimension.  For $1\le p\le k$, set
\[
r=|\Lambda u|,\qquad
\omega=\frac{\Lambda u}{|\Lambda u|},
\qquad
\tilde f_R(u)
=
r^{-(k-1)/2}
\psi\left(\log\frac rR\right)
\bigl(e_p-\omega_p\omega\bigr).
\]
A direct computation in \eqref{eq:anisotropic-q} shows that
$\ind \mathcal{H}_{\bla}=\infty$ whenever
$(7-k)\lambda_p^2>(k-1)\sum_{j\ne p}\lambda_j^2$ for some $p$.
Consequently, $k=2$ is always unstable; for $k=3$, every non-isotropic
parameter is unstable; and for $k=4,5,6$ it suffices, respectively, that
$\lambda_p^2>\sum_{j\ne p}\lambda_j^2$,
$\lambda_p^2>2\sum_{j\ne p}\lambda_j^2$, or
$\lambda_p^2>5\sum_{j\ne p}\lambda_j^2$ for some $p$.
\end{remark}

These are explicit sufficient conditions.  More generally, the
unstable set is open in the set of positive slope parameters.

\begin{proposition}\label{prop:openness}
For every $k\ge1$, the set
\[
\left\{
\bla\in(0,\infty)^k:
\mathcal{H}_{\bla}\text{ is unstable}
\right\}
\]
is open in $(0,\infty)^k$.
\end{proposition}

\begin{proof}
Fix $\bla_0\in(0,\infty)^k$ such that $\mathcal{H}_{\bla_0}$ is unstable, and
write $\Lambda_0=\operatorname{diag}(\bla_0)$.  By
\eqref{eq:stability-cond} there exists
$\tilde f_0\in C_c^\infty(\R^k;\R^k)$ with
$q_{\bla_0}[\tilde f_0]<0$.  Set
\[
f(x)=\tilde f_0(\Lambda_0^{-1}x).
\]
For $\bla\in(0,\infty)^k$ and
$\Lambda=\operatorname{diag}(\bla)$, \eqref{eq:anisotropic-q} gives
\[
\left(\prod_{i=1}^k\lambda_i\right)
q_{\bla}[f\circ\Lambda]
=
\sum_{j=1}^k\lambda_j^2
\int_{\R^k}
h\left[
|\pr_{x_j}f|^2
+
|\pr_{x_j}(x\cdot f)|^2
-
\frac{3(f^j)^2}{h^2}
\right]dx .
\]
The right-hand side is continuous in $\bla$ and negative at $\bla_0$,
because $f\circ\Lambda_0=\tilde f_0$.  It remains negative in a
neighborhood of $\bla_0$, and the positivity of
$\prod_i\lambda_i$ gives $q_{\bla}[f\circ\Lambda]<0$ there.
\end{proof}

Finally, instability persists after adjoining zero slopes and hence, by perturbation, sufficiently small positive slopes.  Indeed, \eqref{eq:L^2-functional} and Proposition~\ref{prop:stability-reduction} remain valid for nonnegative slopes, so $q_{(\bla,0,\ldots,0)}[\cdot]<0$ still implies infinite index.

\begin{proposition}\label{prop:adding-slopes}
Suppose $\mathcal{H}_{\bla}$ is unstable.  For every $m\ge1$, the helicoid
associated with $(\bla,0,\ldots,0)\in\R^{k+m}$ is unstable.  Moreover,
there exists $\epsilon_0>0$ such that the helicoid associated with
$(\bla,\epsilon,\ldots,\epsilon)\in(0,\infty)^{k+m}$ is unstable for
every $0<\epsilon<\epsilon_0$.
\end{proposition}

\begin{proof}
Choose $ f\in C_c^\infty(\R^k;\R^k)$ with
$q_{\bla}[ f]<0$ and fix
$0\ne\chi\in C_c^\infty(\R^m)$.  For $L>0$ set
\[
\chi_L(v)=\chi(v/L),
\qquad
\widehat f_L(u,v)
=
\bigl(\chi_L(v) f(u),0\bigr)
\]
for $(u,v)\in\R^k\times\R^m$.  For
$(\bla,0,\ldots,0)$ the function $h$ is independent of $v$, and direct
substitution in \eqref{eq:L^2-functional} gives
\begin{equation}\label{eq:zero-slope-extension}
q_{(\bla,0,\ldots,0)}[\widehat f_L]
=
\|\chi_L\|_2^2 q_{\bla}[ f]
+
\|\nabla\chi_L\|_2^2
\int_{\R^k}
h\left[
| f|^2
+
\bigl((\Lambda u)\cdot f\bigr)^2
\right]du .
\end{equation}
Since
\[
\|\chi_L\|_2^2=L^m\|\chi\|_2^2,
\qquad
\|\nabla\chi_L\|_2^2
=
L^{m-2}\|\nabla\chi\|_2^2,
\]
the right-hand side of \eqref{eq:zero-slope-extension} is negative for
all sufficiently large $L$.  Fix such an $L$.  For this fixed compactly
supported field,
\[
q_{(\bla,\epsilon,\ldots,\epsilon)}[\widehat f_L]
\longrightarrow
q_{(\bla,0,\ldots,0)}[\widehat f_L]
\qquad\text{as }\epsilon\to0,
\]
and is therefore negative for all sufficiently small $\epsilon>0$.
\end{proof}

\section{Entireness criterion for generalized helicoids}\label{sec:entireness}

Let $\Lambda=\operatorname{diag}(\bla)$ with $\lambda_i>0$. On $\R^{2k+1} = \R^k \times \R^k \times \R$ with coordinates $(x,y,z)$, consider the generalized helicoid $\mathcal{H}_{\bla}\subset\R^{2k+1}$ parametrized by \eqref{eq:general-helicoid}, or equivalently
\begin{equation}\label{eq:generalized-helicoid-real}
    F_{\bla}(u,s)=\bigl(\cos(\Lambda s)u, \, \sin(\Lambda s)u, \, s\bigr), \qquad (u,s)\in\R^k\times\R.
\end{equation}

The sufficient condition obtained in \cite[Theorem 1.8]{TsaiTsuiWanWang} is also necessary.

\begin{theorem}
The helicoid $\mathcal{H}_{\bla}$ is an entire graph over a $(k+1)$-plane if and only if there exists $B\in\R^{k\times k}$ such that
\begin{equation}\label{eq:entire-criterion}
  \det\!\left(\cos(\Lambda s)+B\sin(\Lambda s)\right)>0 \qquad \forall s\in\R.  
\end{equation}
\end{theorem}

\begin{proof}
Let $e_z = (0,0,1)^T \in \R^{2k+1}$ and let $P_\Pi: \R^{2k+1} \to \Pi$ be the orthogonal projection onto a $(k+1)$-plane $\Pi$. The submanifold $\mathcal{H}_{\bla}$ is an entire graph over $\Pi$ if and only if $\Phi_{\Pi,\bla} = P_\Pi \circ F_{\bla} : \R^k \times \R \to \Pi$ is a global diffeomorphism.

Setting $E(s) = (\cos(\Lambda s), \sin(\Lambda s), 0)^T \in \R^{(2k+1)\times k}$, the Jacobian matrix
\[
D\Phi_{\Pi,\bla}(u,s) = \bigl[ P_\Pi E(s), \, P_\Pi e_z + P_\Pi E'(s)u \bigr]
\]
is affine in $u$. Since $\Phi_{\Pi,\bla}$ is a diffeomorphism, $D\Phi_{\Pi,\bla}(u,s)$ is non-singular everywhere, forcing 
\[
P_\Pi E'(s)u \in \operatorname{im}(P_\Pi E(s)), \qquad \forall u \in \R^k.
\]
Since $\lambda_i>0$, the columns of $E(s)$ and $E'(s)$ span $\{z=0\}$; hence
\[
 \dim P_\Pi(\{z=0\})\le k.
\]
If $e_z \notin \Pi$, then $\dim(\Pi^\perp \cap \{z=0\}) = k-1$, yielding $\dim P_\Pi(\{z=0\}) = 2k - (k-1) = k+1$, a contradiction. Therefore, $e_z \in \Pi$ and we can write $\Pi = W \times \R$ with $W = \Pi \cap \{z=0\} \subset \R^{2k}$.

The map $\Phi_{\Pi,\bla}$ is a diffeomorphism if and only if $P_W \begin{pmatrix}\cos(\Lambda s) \\ \sin(\Lambda s)\end{pmatrix} : \R^k \to W$ is invertible for all $s \in \R$. Invertibility at $s=0$ implies $W = \{ (x, B^T x) \mid x \in \R^k \}$ for a unique $B \in \R^{k \times k}$, whose orthogonal complement in $\{z=0\}$ is $W^\perp = \{ (v_1, v_2) \in \R^{2k} \mid v_1 + B v_2 = 0 \}$.

Thus, $P_W \begin{pmatrix}\cos(\Lambda s) \\ \sin(\Lambda s)\end{pmatrix}$ has a non-trivial kernel if and only if $\begin{pmatrix}\cos(\Lambda s) \\ \sin(\Lambda s)\end{pmatrix} u \in W^\perp$ for some $u \neq 0$, which is equivalent to $(\cos(\Lambda s) + B\sin(\Lambda s))u = 0$. Hence, invertibility holds for all $s \in \R$ if and only if $\det(\cos(\Lambda s) + B\sin(\Lambda s)) \ne 0$. Because $\det I_k = 1 > 0$, continuity implies that non-vanishing is equivalent to positivity for all $s \in \R$.
\end{proof}

\begin{corollary}
The helicoid $H_k$ is an entire graph over a $(k+1)$-plane if and only if $k$ is even.
\end{corollary}

\begin{proof}
For $H_k$, one has $\Lambda=I_k$. If $k$ is even, entireness follows from the construction in \cite[Theorem~1.8]{TsaiTsuiWanWang}. If $k$ is odd, the determinant condition \eqref{eq:entire-criterion} in the preceding theorem fails at $s=\pi$, since $\det\!\left(\cos(\pi I_k)+B\sin(\pi I_k)\right) =\det(-I_k)=(-1)^k=-1$.
\end{proof}


\begin{thebibliography}{99}

\bibitem{BarbosaDajczerJorge}
J.~L.~M. Barbosa, M. Dajczer, and L.~P.~M. Jorge,
\newblock Minimal ruled submanifolds in spaces of constant curvature,
\newblock \emph{Indiana Univ. Math. J.} \textbf{33} (1984), no.~4,
531--547,
\newblock \href{https://doi.org/10.1512/iumj.1984.33.33028}
{doi:10.1512/iumj.1984.33.33028}.

\bibitem{BernsteinBreiner}
J.~Bernstein and C.~Breiner,
\newblock Conformal structure of minimal surfaces with finite topology,
\newblock \emph{Comment. Math. Helv.} \textbf{86} (2011), no.~2,
353--381,
\newblock \href{https://doi.org/10.4171/CMH/226}
{doi:10.4171/CMH/226}.

\bibitem{BryantAustere}
R.~L. Bryant,
\newblock Some remarks on the geometry of austere manifolds,
\newblock \emph{Bol. Soc. Brasil. Mat. (N.S.)} \textbf{21} (1991), no.~2,
133--157,
\newblock \href{https://doi.org/10.1007/BF01237361}
{doi:10.1007/BF01237361}.

\bibitem{ColdingMinicozzi}
T.~H. Colding and W.~P. Minicozzi~II,
\newblock The space of embedded minimal surfaces of fixed genus in a 3-manifold,
\newblock \emph{Ann. of Math. (2)} \textbf{160} (2004):
I. Estimates off the axis for disks, no.~1, 27--68,
\href{https://doi.org/10.4007/annals.2004.160.27}
{doi:10.4007/annals.2004.160.27};
II. Multi-valued graphs in disks, no.~1, 69--92,
\href{https://doi.org/10.4007/annals.2004.160.69}
{doi:10.4007/annals.2004.160.69};
III. Planar domains, no.~2, 523--572,
\href{https://doi.org/10.4007/annals.2004.160.523}
{doi:10.4007/annals.2004.160.523};
IV. Locally simply connected, no.~2, 573--615,
\href{https://doi.org/10.4007/annals.2004.160.573}
{doi:10.4007/annals.2004.160.573}.

\bibitem{dCP}
M.~P. do Carmo and C.~K. Peng,
\newblock Stable complete minimal surfaces in $\mathbb{R}^3$ are planes,
\newblock \emph{Bull. Amer. Math. Soc. (N.S.)} \textbf{1} (1979), no.~6,
903--906,
\newblock \href{https://doi.org/10.1090/S0273-0979-1979-14689-5}
{doi:10.1090/S0273-0979-1979-14689-5}.

\bibitem{FCS}
D. Fischer-Colbrie and R. Schoen,
\newblock The structure of complete stable minimal surfaces in
3-manifolds of non-negative scalar curvature,
\newblock \emph{Comm. Pure Appl. Math.} \textbf{33} (1980), no.~2,
199--211,
\newblock \href{https://doi.org/10.1002/cpa.3160330206}
{doi:10.1002/cpa.3160330206}.

\bibitem{HarveyLawson}
R. Harvey and H.~B. Lawson, Jr.,
\newblock Calibrated geometries,
\newblock \emph{Acta Math.} \textbf{148} (1982), 47--157,
\newblock \href{https://doi.org/10.1007/BF02392726}
{doi:10.1007/BF02392726}.

\bibitem{KerckhoveLawlor}
M. Kerckhove and G.~R. Lawlor,
\newblock A family of stratified area-minimizing cones,
\newblock \emph{Duke Math. J.} \textbf{96} (1999), no.~2, 401--424,
\newblock \href{https://doi.org/10.1215/S0012-7094-99-09612-6}
{doi:10.1215/S0012-7094-99-09612-6}.

\bibitem{KrainesQuaternionic}
V.~Y. Kraines,
\newblock Topology of quaternionic manifolds,
\newblock \emph{Trans. Amer. Math. Soc.} \textbf{122} (1966), no.~2,
357--367,
\newblock \href{https://doi.org/10.1090/S0002-9947-1966-0192513-X}
{doi:10.1090/S0002-9947-1966-0192513-X}.

\bibitem{LawlorCriterion}
G.~R. Lawlor,
\newblock A sufficient criterion for a cone to be area-minimizing,
\newblock \emph{Mem. Amer. Math. Soc.} \textbf{91} (1991), no.~446,
vi+111 pp.,
\newblock \href{https://doi.org/10.1090/memo/0446}
{doi:10.1090/memo/0446}.

\bibitem{LawlorDirectedSlicing}
G.~R. Lawlor,
\newblock Proving area minimization by directed slicing,
\newblock \emph{Indiana Univ. Math. J.} \textbf{47} (1998), no.~4,
1547--1592,
\newblock \href{https://doi.org/10.1512/iumj.1998.47.1341}
{doi:10.1512/iumj.1998.47.1341}.

\bibitem{LiGA}
P.~Li, \emph{Geometric Analysis},
Cambridge University Press, Cambridge, 2012.

\bibitem{MeeksRosenberg}
W.~H. Meeks III and H. Rosenberg,
\newblock The uniqueness of the helicoid,
\newblock \emph{Ann. of Math. (2)} \textbf{161} (2005), no.~2, 727--758,
\newblock \href{https://doi.org/10.4007/annals.2005.161.727}
{doi:10.4007/annals.2005.161.727}.

\bibitem{MorganCalibrations}
F. Morgan,
\newblock Calibrations and new singularities in area-minimizing surfaces: A survey,
\newblock in \emph{Variational Methods}, Progr. Nonlinear Differential
Equations Appl., vol.~4, Birkh\"auser, Boston, 1990, pp.~329--342,
\newblock \href{https://doi.org/10.1007/978-1-4757-1080-9_23}
{doi:10.1007/978-1-4757-1080-9\_23}.

\bibitem{SimonGMT}
L.~Simon, \emph{Lectures on Geometric Measure Theory},
Proceedings of the Centre for Mathematical Analysis, Australian National University, vol.~3, Australian National University, Centre for Mathematical Analysis, Canberra, 1983.

\bibitem{TsaiTsuiWanWang}
C.-J. Tsai, M.-P. Tsui, J. Wan, and M.-T. Wang,
\newblock Constructing entire minimal graphs by evolving planes,
\newblock \emph{Calc. Var. Partial Differential Equations} \textbf{65}
(2026), no.~6, article no.~196,
\newblock \href{https://doi.org/10.1007/s00526-026-03372-8}
{doi:10.1007/s00526-026-03372-8}.

\end{thebibliography}
\end{document}